\documentclass[11pt]{article}

\usepackage{fancyhdr}
\usepackage{amsfonts}
\usepackage{amsmath}
\usepackage{amssymb}
\usepackage{amsthm}
\usepackage{tikz}
\usepackage{amscd}
\usepackage{amstext}
\usepackage{verbatim}

\usepackage[draft=false,colorlinks,bookmarksnumbered,linkcolor=black, citecolor=black]{hyperref}

\usepackage{oldgerm}
\usepackage{graphics}
\usepackage{graphicx}

\newlength{\fixboxwidth}
\newcommand*{\abs}[1]{\left| #1 \right|}                                
\newcommand*{\norm}[1]{\left\| #1 \right\|}                             
\newcommand*{\sep}{\; \vrule \;}                                        
\newcommand{\loc}{\mathrm{loc}}											

\newcommand{\re}{\mathbb{R}}\newcommand{\N}{\mathbb{N}}
\newcommand{\zz}{\mathbb{Z}}

\newcommand{\com}{\mathbb{C}}

\newcommand{\R}{{\re}^d}

\newcommand{\cs}{{\mathcal S}}

\newcommand{\cl}{{\mathcal L}}

\newcommand{\cf}{{\mathcal F}}
\newcommand{\cfi}{{\cf}^{-1}}

\newcommand{\supp}{{\rm supp \, }}

\renewcommand{\qedsymbol}{$\blacksquare$}

\newcommand{\be}{\begin{equation}}
\newcommand{\ee}{\end{equation}}
\newcommand{\beq}{\begin{eqnarray}}
\newcommand{\beqq}{\begin{eqnarray*}}
\newcommand{\eeq}{\end{eqnarray}}
\newcommand{\eeqq}{\end{eqnarray*}}

\newtheorem{satz}{Theorem}[section]

\newtheorem{defi}{Definition}[section]
\newtheorem{lem}{Lemma}[section]
 
\newtheorem{prop}{Proposition}[section]

\numberwithin{equation}{section}

\begin{document}


\title{Hölder Inequalities for Besov-Morrey and Triebel-Lizorkin-Morrey Spaces}
\author{Marc Hovemann\footnote{Friedrich-Schiller-Universität Jena, Faculty of Mathematics and Computer Science, Inselplatz 5, 07743 Jena, Germany. 
Email: \href{mailto:marc.hovemann@uni-jena.de }{marc.hovemann@uni-jena.de } }
}

\date{\today}

\maketitle

\noindent\textbf{Abstract:}  In this paper we prove Hölder inequalities for the Besov-Morrey spaces $  \mathcal{N}^{s}_{u,p,q}(\mathbb{R}^{d})   $ and the Triebel-Lizorkin-Morrey spaces $   \mathcal{E}^{s}_{u,p,q}(\mathbb{R}^{d})  $. We observe that these Morrey smoothness spaces are pointwise multiplier spaces if certain conditions concerning the parameters also including $ s > d \max ( 0, \frac{1}{p} - 1)   $ are fulfilled. In this context we significantly generalize the Hölder inequalities for the original Besov and Triebel-Lizorkin spaces obtained by Sickel and Triebel in 1995. For the proofs we use paramultiplication and some Franke-Jawerth embeddings obtained by Haroske and    Skrzypczak.

\vspace{0,2 cm}


\vspace{0,2 cm}


\section{Introduction and Main Results}

Nowadays Besov spaces $B^s_{p,q} (\R)$ and Triebel-Lizorkin spaces $F^s_{p,q} (\R)$ are well-established tools to describe the regularity of functions and distributions. The Besov spaces have been introduced by Nikol'skij \cite{Ni2} and Besov \cite{Be1959}, \cite{Be1961} between 1951 and 1961, whereas the Triebel-Lizorkin spaces have been introduced around 1970 by Lizorkin \cite{Liz1, Liz2} and Triebel \cite{Tr73}.  Later these function spaces have been investigated in detail in the books of Triebel~\cite{Tr83,Tr92,Tr06} which contain numerous other historical references. 
In recent years a growing number of authors works with further generalizations of these function spaces defined upon \emph{Morrey spaces} instead of Lebesgue spaces $L_p$. In connection with that Besov-Morrey spaces $  \mathcal{N}^{s}_{u,p,q}(\mathbb{R}^d) $ with $ 0 < p \leq u < \infty$, $ 0 < q \leq \infty $ and $ s \in \mathbb{R} $ showed up. They have been introduced by Kozono and Yamazaki \cite{KoYa} in 1994. On the other hand, Triebel-Lizorkin-Morrey spaces $ \mathcal{E}^{s}_{u,p,q}(\R) $ with $ 0 < p \leq u < \infty$, $ 0 < q \leq \infty $, $ s \in \mathbb{R} $ and Triebel-Lizorkin-type spaces $ F^{s,\tau}_{p,q}(\R) $ with $ 0 < p < \infty $, $ 0 < q \leq \infty $, $ s \in \mathbb{R} $, $ 0\leq \tau<\infty $ attracted a lot of attention. 
The spaces $ \mathcal{E}^{s}_{u,p,q}(\R) $ have been investigated for the first time by Tang and Xu in 2005, see \cite{TangXu}, while  
$F^{s,\tau}_{p,q}(\R) $ showed up for the first time in 2008 in some papers of Yang and Yuan~\cite{yy1,yy2}.  Later on, using a different notation, the latter also appeared in~\cite{Tr14}. 
Although the spaces $ \mathcal{E}^{s}_{u,p,q}(\R) $ and  $ F^{s,\tau}_{p,q}(\R) $ are defined quite differently, they have a lot of properties in common. 
Moreover, under certain conditions on the parameters they even coincide \cite{ysy}.

In this paper we want to prove Hölder inequalities for the Besov-Morrey spaces and the Triebel-Lizorkin-Morrey spaces. More precisely, we want to identify parameter constellations such that pointwise multiplication inequalities of the form 
\begin{equation}\label{eq_intro_task1}
\Vert f \cdot g \vert  \mathcal{A}^{s}_{u,p,q}(\mathbb{R}^d)   \Vert   \lesssim  \Vert  f   \vert   \mathcal{A}^{s}_{u_{1},p_{1},q_{1}}(\mathbb{R}^d) \Vert     \Vert g \vert  \mathcal{A}^{s}_{u_{2},p_{2},q_{2}}(\mathbb{R}^d)   \Vert
\end{equation}
with $\mathcal{A} \in \{ \mathcal{N} , \mathcal{E}  \}$ hold. There, the smoothness $s$ is fixed. Such inequalities play an important role within the solution theory of PDEs. For example, in \cite{BaaSch19}, \cite{BaaSch23} and \cite{BaaSchTr25} forerunners of \eqref{eq_intro_task1} have been used to prove the  existence and uniqueness of mild and strong solutions for fractional  nonlinear heat equations and hyperdissipative Navier-Stokes equations. However, in the aforementioned references only versions of \eqref{eq_intro_task1} for the original Besov and Triebel-Lizorkin spaces or for $ \mathcal{A}^{s}_{u,p,q}(\mathbb{R}^d)   $ with additional restrictive conditions concerning the smoothness parameter $s$ have been used. Therefore, inequalities of the form \eqref{eq_intro_task1} that hold for a large range of the parameters are in high demand. 

As already mentioned, versions of \eqref{eq_intro_task1} for the original Besov spaces $B^s_{p,q} (\R)$ and the original Triebel-Lizorkin spaces $F^s_{p,q} (\R)$ are well-known since many years. For instance, the following result concerning Hölder inequalities has been obtained by Sickel and Triebel already in 1995, see \cite[Theorem 4.2.1]{SiTri}. 

\begin{satz}\label{thm_hölder_historical1}
Let $s > 0$, $0 < p_{1} < \infty$, $0 < p_{2} < \infty $,  $0 < p  < \infty $, $ 0 < q_{1} \leq \infty    $, $ 0 < q_{2} \leq \infty    $  and  $ 0 < q \leq \infty    $. Let 
\begin{equation}\label{eq_hölder_meet_franke1_hist}
\frac{1}{r_{1}} = \frac{1}{p_{1}} - \frac{s}{d} > 0, \qquad \frac{1}{r_{2}} = \frac{1}{p_{2}} - \frac{s}{d} > 0, \qquad \frac{1}{r_{1}} + \frac{1}{r_{2}} = \frac{1}{r} = \frac{1}{p} - \frac{s}{d} < 1 .
\end{equation}
\begin{itemize}
\item[(i)] Then it holds
\begin{equation}\label{eq_main_N_asequation_histB}
B^{s}_{p_{1},q_{1}}(\mathbb{R}^d) B^{s}_{p_{2},q_{2}}(\mathbb{R}^d) \hookrightarrow B^{s}_{p,q}(\mathbb{R}^d) ,
\end{equation}
if and only if
\begin{equation}\label{eq_main_N_cond_qq1q2_histB}
0 < q_{1} \leq r_{1}, \qquad 0 < q_{2} \leq r_{2}, \qquad \max(q_{1},q_{2}) \leq q \leq \infty .
\end{equation}

\item[(ii)] In addition, it holds
\begin{equation}\label{eq_main_E_asequation_histF}
F^{s}_{p_{1},q_{1}}(\mathbb{R}^d) F^{s}_{p_{2},q_{2}}(\mathbb{R}^d) \hookrightarrow F^{s}_{p,q}(\mathbb{R}^d) ,
\end{equation}
if and only if
\begin{equation}\label{eq_main_E_cond_qq1q2_histF}
 \max(q_{1},q_{2}) \leq q \leq \infty .
\end{equation}
\end{itemize}
\end{satz}
Since it holds $\mathcal{E}^{s}_{p,p,q}(\mathbb{R}^{d}) = F^{s}_{p,q}(\R)$ and $\mathcal{N}^{s}_{p,p,q}(\mathbb{R}^{d}) = B^{s}_{p,q}(\R)$, Theorem \ref{thm_hölder_historical1} can be interpreted as a special case of \eqref{eq_intro_task1}. For the proof of Theorem \ref{thm_hölder_historical1} in \cite{SiTri} so-called paramultiplication and a special splitting for products of the form $f \cdot g$ has been used. This proof technique has been invented independently by Peetre \cite{Pee76} and Triebel \cite{Tr77}, see \cite[Remark 6.3]{ysy}. Now we turn to the case of the much more general Besov-Morrey spaces $ \mathcal{N}^{s}_{u,p,q}(\R) $ and Triebel-Lizorkin-Morrey spaces $ \mathcal{E}^{s}_{u,p,q}(\R) $. For $ \mathcal{E}^{s}_{u,p,q}(\R) $ an early result in the spirit of \eqref{eq_intro_task1} can be found in \cite[Theorem 6.3]{ysy}, see also \cite[Theorem 3.60]{Tr14}. There it has been proved that the spaces $ \mathcal{E}^{s}_{u,p,q}(\R) $ are algebras with respect to pointwise multiplication. In both references the additional assumption $ \mathcal{E}^{s}_{u,p,q}(\R) \hookrightarrow L_{\infty} (\mathbb{R}^d)   $ has been used. However, by \cite[Proposition 2.7]{HaSk2014} this implies $  s > \frac{d}{u}   $. Below, in this paper we prove an advanced version of \eqref{eq_intro_task1} that holds for 
\begin{align*}
s > d \max \Big ( 0, \frac{1}{p} - 1  \Big ) .
\end{align*} 
Consequently, in particular for $ p \geq 1   $ we can cover a much larger range of the parameters. The main result of this paper reads as follows.

\begin{satz}\label{thm_main_hölder1}
Let $s > 0$, $0 < p_{1} \leq u_{1} < \infty$, $0 < p_{2} \leq u_{2} < \infty $,  $0 < p \leq u < \infty $, $ 0 < q_{1} \leq \infty    $, $ 0 < q_{2} \leq \infty    $  and  $ 0 < q \leq \infty    $. Let 
\begin{equation}\label{eq_hölder_meet_franke1}
\frac{1}{r_{1}} = \frac{1}{p_{1}} - \frac{s}{d} > 0, \qquad \frac{1}{r_{2}} = \frac{1}{p_{2}} - \frac{s}{d} > 0, \qquad \frac{1}{r_{1}} + \frac{1}{r_{2}} = \frac{1}{r} = \frac{1}{p} - \frac{s}{d} < 1
\end{equation}
and
\begin{equation}\label{eq_hölder_meet_franke2}
\frac{1}{v_{1}} = \frac{1}{u_{1}} - \frac{s}{d} > 0, \qquad \frac{1}{v_{2}} = \frac{1}{u_{2}} - \frac{s}{d} > 0, \qquad \frac{1}{v_{1}} + \frac{1}{v_{2}} = \frac{1}{v} = \frac{1}{u} - \frac{s}{d} < 1 .
\end{equation}
\begin{itemize}
\item[(i)] Then it holds
\begin{equation}\label{eq_main_N_asequation}
\mathcal{N}^{s}_{u_{1},p_{1},q_{1}}(\mathbb{R}^d) \mathcal{N}^{s}_{u_{2},p_{2},q_{2}}(\mathbb{R}^d) \hookrightarrow \mathcal{N}^{s}_{u,p,q}(\mathbb{R}^d) ,
\end{equation}
if
\begin{equation}\label{eq_main_N_cond_qq1q2}
0 < q_{1} \leq r_{1}, \qquad 0 < q_{2} \leq r_{2}, \qquad \max(q_{1},q_{2}) \leq q \leq \infty .
\end{equation}

\item[(ii)] Moreover, it holds
\begin{equation}\label{eq_main_E_asequation}
\mathcal{E}^{s}_{u_{1},p_{1},q_{1}}(\mathbb{R}^d) \mathcal{E}^{s}_{u_{2},p_{2},q_{2}}(\mathbb{R}^d) \hookrightarrow \mathcal{E}^{s}_{u,p,q}(\mathbb{R}^d) ,
\end{equation}
if
\begin{equation}\label{eq_main_E_cond_qq1q2}
 \max(q_{1},q_{2}) \leq q \leq \infty .
\end{equation}
\end{itemize}
\end{satz}

Theorem \ref{thm_main_hölder1} provides the so-called Hölder inequalities for $ \mathcal{N}^{s}_{u,p,q}(\R) $ and $ \mathcal{E}^{s}_{u,p,q}(\R) $. There, \eqref{eq_main_N_asequation} and \eqref{eq_main_E_asequation} can be interpreted as \eqref{eq_intro_task1} with $\mathcal{A} \in \{ \mathcal{N} , \mathcal{E}  \}$. In the special case $  p_{1} = u_{1} $, $ p_{2} = u_{2} $ and  $ p = u  $ the assumptions \eqref{eq_hölder_meet_franke1} and \eqref{eq_hölder_meet_franke2} completely coincide. Then it holds $r_{1} = v_{1}$, $r_{2} = v_{2}$ and $r = v$. Since then also \eqref{eq_hölder_meet_franke1} and \eqref{eq_hölder_meet_franke1_hist} coincide, Theorem \ref{thm_main_hölder1} can be seen as a generalization of Theorem \ref{thm_hölder_historical1}. We provide some further explanations concerning Theorem \ref{thm_main_hölder1}, that at least partly also can be found in Figure \ref{fig_Figure1} below. First of all, we observe that as a consequence of \eqref{eq_hölder_meet_franke1} and \eqref{eq_hölder_meet_franke2} the range of parameters is restricted to
\begin{align*}
\frac{d}{p}  \geq \frac{d}{u} > s > \sigma_{p} \geq \sigma_{u} .
\end{align*} 
If $s$ approaches zero and $q_{1} = q_{2} = q = 2$, then \eqref{eq_main_E_asequation} tends to the classical Hölder inequality for Morrey spaces  
\begin{align*}
\Vert f \cdot g \vert  \mathcal{M}^{v}_{r}(\mathbb{R}^d)   \Vert   \lesssim  \Vert  f   \vert   \mathcal{M}^{v_{1}}_{r_{1}}(\mathbb{R}^d) \Vert     \Vert g \vert  \mathcal{M}^{v_{2}}_{r_{2}}(\mathbb{R}^d)   \Vert
\end{align*}
with 
\begin{align*}
 \frac{1}{r} = \frac{1}{r_{1}} + \frac{1}{r_{2}} < 1 \qquad \mbox{and} \qquad  \frac{1}{v} = \frac{1}{v_{1}} + \frac{1}{v_{2}} < 1  \qquad \mbox{and} \qquad r \leq v, \quad r_{1} \leq v_{1}, \quad r_{2} \leq v_{2}. 
\end{align*}
The interrelation of the parameters $r, r_{1}, r_{2}$ also is illustrated by the axis $s = 0$ in Figure \ref{fig_Figure1}. For Morrey smoothness spaces with positive smoothness $ s > 0 $ the Hölder inequality is shifted along lines with slope $d$ to the smoothness level $s$. Then the parameters are connected in the sense 
\begin{equation}\label{eq_hölder_meet_franke1_explain1}
\frac{1}{r_{1}} + \frac{s}{d} = \frac{1}{\tilde{p}_{1}} , \qquad \frac{1}{r_{2}} + \frac{s}{d} = \frac{1}{\tilde{p}_{2}} , \qquad  \frac{1}{r} + \frac{s}{d} = \frac{1}{\tilde{p}} 
\end{equation}
and
\begin{equation}\label{eq_hölder_meet_franke2_explain2}
\frac{1}{v_{1}} + \frac{s}{d}  = \frac{1}{\tilde{u}_{1}} , \qquad \frac{1}{v_{2}} + \frac{s}{d} = \frac{1}{\tilde{u}_{2}} , \qquad  \frac{1}{v} + \frac{s}{d} = \frac{1}{\tilde{u}} .
\end{equation}
Here the relations \eqref{eq_hölder_meet_franke1_explain1} are illustrated in Figure \ref{fig_Figure1}. In order to depict \eqref{eq_hölder_meet_franke2_explain2} one can think of an $ (\frac{1}{u},s)-    $diagram similar to Figure \ref{fig_Figure1} where $  p, r_{1}, r_{2}, r,    \tilde{p}_{1}, \tilde{p}_{2} , \tilde{p}   $ are replaced by $  u, v_{1}, v_{2}, v,    \tilde{u}_{1}, \tilde{u}_{2} , \tilde{u}   $. Finally, let us mention that Figure \ref{fig_Figure1} has some similarities with \cite[Page 106, Figure 1]{SiTri}, where the Hölder inequalities for the classical Besov and Triebel-Lizorkin spaces have been discussed. 

\begin{figure}[h]
\centering
\begin{tikzpicture}[thick]
\draw[->] (-2,0) -- (10,0) ;
\draw[->] (-2,0) -- (-2,6) ;
\draw (-2,0) -- (1,6) ;
\draw (7,-0.1) -- (7,0.1) ;

\draw[dotted] (0,0) -- (2,4) ;
\draw[dotted] (4,0) -- (6,4) ;
\draw[dotted] (6,0) -- (8,4) ;

\draw[dotted] (-2,4) -- (10,4) ;

\draw (7,0) -- (10,6) ;



\node at (-2.3,5.8) {$s$} ;
\node at (1.8,6) {$s= \frac{d}{p}$} ;
\node at (11.1,6) {$s=  \frac{d}{p} - d$} ;
\node at (7,-0.5) {$1$} ;
\node at (-1.8,-0.3) {$0$} ;
\node at (10,-0.5) {$ \frac{1}{p}$} ;

\node at (0,-0.5) {$ \frac{1}{r_{1}}$} ;
\node at (4,-0.5) {$ \frac{1}{r_{2}}$} ;
\node at (6,-0.5) {$ \frac{1}{r}$} ;

\draw (0,0) circle (0.1cm) ;
\draw (4,0) circle (0.1cm) ;
\draw (6,0) circle (0.1cm) ;

\node at (2,4.5) {$ \frac{1}{\tilde{p}_{1}}$} ;
\node at (6,4.5) {$ \frac{1}{\tilde{p}_{2}}$} ;
\node at (8,4.5) {$ \frac{1}{\tilde{p}}$} ;

\draw (2,4) circle (0.1cm) ;
\draw (6,4) circle (0.1cm) ;
\draw (8,4) circle (0.1cm) ;

\end{tikzpicture}
\caption{Parameter constellation in Theorem \ref{thm_main_hölder1}.}
\label{fig_Figure1}
\end{figure}
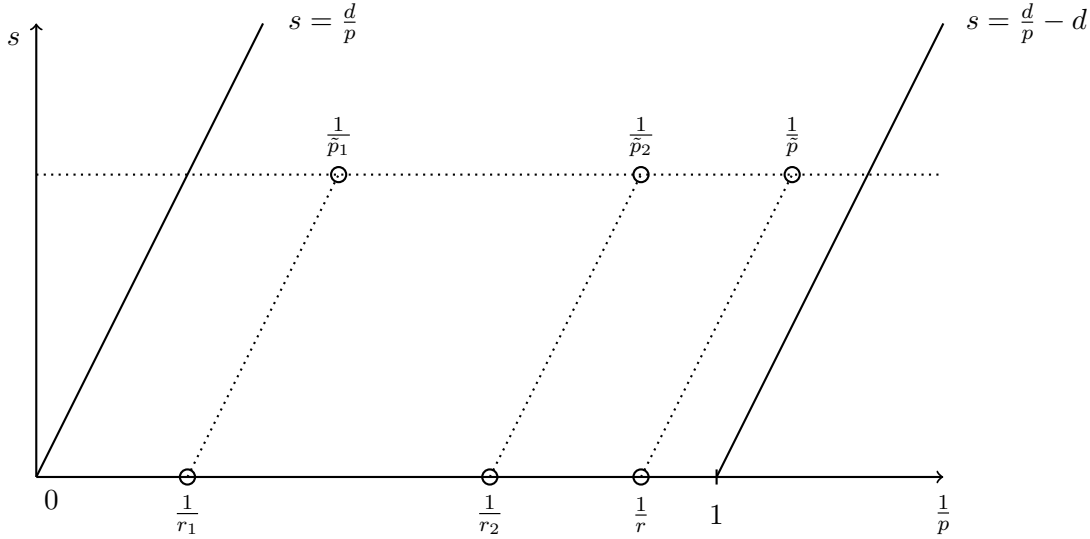

It is the main goal of this paper to prove Theorem \ref{thm_main_hölder1}. For that purpose in Section \ref{sect:prelim} we recall some basic properties of the Besov-Morrey spaces and the Triebel-Lizorkin-Morrey spaces. In particular, we collect some Franke-Jawerth embeddings obtained by Haroske and Skrzypczak in \cite{HaSk2014}. In Section \ref{sec_paramultiplication} we recall the concept of paramultiplication and prove some Fourier multiplier theorems for the function spaces of interest. Finally, in Section \ref{sec_proof_mainresult} we put everything together in order to prove Theorem \ref{thm_main_hölder1}. However, first of all we fix some notation.

\medskip

\noindent\textbf{Notation:} As usual, $\N$ denotes the natural numbers, $\N_0:=\N\cup\{0\}$, $\zz$ describes the integers and~$\re$ the real numbers. Further, $\R$ with $d\in\N$ denotes the $d$-dimensional Euclidean space and we put
$$
    B(x,t) := \left\{ y \in \R : \abs{x-y}< t \right\}\, , \qquad x \in \R,\,\; t>0.
$$
All functions are assumed to be complex-valued. That means we consider functions $f\colon \R \to \com$.  $\mathcal{S}(\R)$ is the collection of all Schwartz functions on $\R$ endowed with the usual topology and by $\mathcal{S}'(\R)$ we denote its topological dual, the space of all bounded linear functionals on~$\mathcal{S}(\R)$ equipped with the weak-$\ast$ topology. 
The symbol $\cf$ refers to the Fourier transform and $\cfi$ to its inverse, both defined on $\cs'(\R)$. 
For domains (open connected sets) $\Omega\subseteq\R$ and $0<v\leq \infty$ by $L_v^\loc(\Omega)$ we mean the set of locally $v$-integrable (or locally essentially bounded) functions on $\Omega$. Furthermore, $\mathcal{D}(\Omega)=C_0^\infty(\Omega)$ denotes the set of infinitely often differentiable functions with compact support on $\Omega$. 
Its topological dual $\mathcal{D}'(\Omega)$ is the space of distributions on $\Omega$.
Almost all function spaces considered in this paper are subspaces of regular distributions from $\cs'(\R)$, interpreted as spaces of equivalence classes with respect to almost everywhere equality. 
Given two quasi-Banach spaces $X$ and $Y$, the norm of a linear operator $T\colon X\to Y$ is denoted by $\norm{T \sep \cl (X,Y)}$. 
Moreover, we write $X \hookrightarrow Y$ if the natural embedding of $X$ into $Y$ is continuous. 
For $0<p<\infty$ and $0<q\leq \infty$ we shall use the well-established quantities
$$
   \sigma_p:= d\,  \max\!\left\{0, \frac 1p - 1\right\} \qquad \text{and}\qquad 
    \sigma_{p,q}:= d\, \max\!\left\{0, \frac 1p -1 , \frac 1q - 1 \right\}.
$$
The symbols $C, C_1, c, c_{1}, \ldots$ denote positive constants depending only on the fixed parameters  and probably on auxiliary functions. 
Unless otherwise stated their values may vary from line to line. 
With $A \lesssim B$ we mean $ A \leq C B  $ for a constant $ C > 0 $ independent of $A$ and $B$. The notation $ A \sim B $ stands for $A \lesssim B$ and $B \lesssim A$.

\section{Besov-Morrey and Triebel-Lizorkin-Morrey Spaces: Definitions and Basic Properties}\label{sect:prelim}

Besov-Morrey spaces $ \mathcal{N}^{s}_{u,p,q}(\R)$ and Triebel-Lizorkin-Morrey spaces $ \mathcal{E}^{s}_{u,p,q}(\R)$ are built upon Morrey spaces $\mathcal{M}^{u}_{p}(\R)$. Therefore, at first we recall the definition of the latter. 

\begin{defi} \label{def_mor}
Let $ 0 < p \leq u < \infty$. Then the Morrey space $\mathcal{M}^{u}_{p}(\R)$ is the collection of all functions $ f \in L_{p}^{\loc}(\R) $ such that
    \begin{align*}
        \Vert f \vert \mathcal{M}^{u}_{p}(\R) \Vert
        := \sup_{y \in \R, r > 0} \vert B(y,r) \vert^{\frac{1}{u}-\frac{1}{p}} \Big ( \int_{B(y,r)} \abs{f(x)}^{p} dx  \Big )^{\frac{1}{p}} < \infty.
    \end{align*} 
\end{defi}
The Morrey spaces $ \mathcal{M}^{u}_{p}(\R)$ are known to be quasi-Banach spaces. For $ p \geq 1$ they are Banach spaces.  They have many connections to ordinary Lebesgue spaces $ L_{p}(\R)$.  Indeed, for $ 0 < p_{2} \leq p_{1} \leq u < \infty $ we have
\begin{align*}
    L_{u}(\R) 
    = \mathcal{M}^{u}_{u}(\R) 
    \hookrightarrow \mathcal{M}^{u}_{p_{1}}(\R)
    \hookrightarrow \mathcal{M}^{u}_{p_{2}}(\R).
\end{align*} 
In order to define the Besov-Morrey spaces $ \mathcal{N}^{s}_{u,p,q}(\R)$ and the Triebel-Lizorkin-Morrey spaces  $\mathcal{E}^{s}_{u,p,q}(\R)$ we need a so-called smooth dyadic decomposition of unity. Let $\varphi_0 \in C_0^{\infty}({\R})$ be a non-negative function such that $\varphi_0(x) = 1$ if $\abs{x}\leq 1$ and $ \varphi_0 (x) = 0$ if $\abs{x}\geq \frac{3}{2}$. 
For $k\in \N$ we define
$$
    \varphi_k(x) := \varphi_0(2^{-k}x) - \varphi_0(2^{-k+1}x),\qquad\ x \in \R. 
$$
Then 
$$
    \sum_{k=0}^\infty \varphi_k(x) = 1, \qquad x \in \R, 
$$
and
\begin{equation}\label{eq_supp_varphik}
    \supp \varphi_k \subset \Big \{x\in \R \sep 2^{k-1}\le \abs{x} \le 2^{k+1}\Big \}, \qquad k \in \N,
\end{equation}
which justifies the name smooth dyadic decomposition of unity for the system $(\varphi_k)_{k\in \N_0 }$.  Moreover, using the Paley-Wiener-Schwartz theorem \cite[Theorem 2 in Chapter 1.2.1]{Tr83} we find that for all $f\in\mathcal{S}'(\R)$ the distributions $\cfi[\varphi_{k}\, \cf f]\in\mathcal{S}'(\R)$ with $k\in\N_0$ are actually smooth functions on $\R$. 
This allows for the following definition of Besov-Morrey and Triebel-Lizorkin-Morrey spaces.

\begin{defi}\label{def_bms}
Let $ s \in \mathbb{R}$, $ 0 < p \leq u < \infty $ and $ 0 < q \leq \infty $. Further, let $ (\varphi_{k})_{k\in \N_0 }$ be a smooth dyadic decomposition of unity. Then the Besov-Morrey space $  \mathcal{N}^{s}_{u,p,q}(\mathbb{R}^{d}) $ is the collection of all distributions $ f \in \mathcal{S}'(\mathbb{R}^{d})$ such that
\begin{align*} 
    \Vert f \vert \mathcal{N}^{s}_{u,p,q}(\mathbb{R}^{d}) \Vert
    :=  \Big ( \sum_{k = 0}^{\infty} 2^{ksq}   \Vert \mathcal{F}^{-1}[\varphi_{k} \mathcal{F}f]  \vert \mathcal{M}^{u}_{p}(\R)   \Vert^{q} \Big )^{\frac{1}{q}} < \infty .
\end{align*}
In the case $ q = \infty $ the usual modifications are made.
\end{defi}

\begin{defi}
\label{def_tlm}
    Let $ s \in \mathbb{R}$,  $ 0 < p \leq u < \infty$ and $0 < q \leq \infty$. Further, let $ (\varphi_{k})_{k\in \N_0 }$ be a smooth dyadic decomposition of unity. Then the Triebel-Lizorkin-Morrey space $  \mathcal{E}^{s}_{u,p,q}(\mathbb{R}^{d})$ collects all $ f \in \mathcal{S}'(\mathbb{R}^{d})$ for that
    \begin{align*} 
        \Vert f \vert \mathcal{E}^{s}_{u,p,q}(\mathbb{R}^{d}) \Vert := \Big \Vert  \Big ( \sum_{k = 0}^{\infty} 2^{ksq} \vert \mathcal{F}^{-1}[\varphi_{k} \mathcal{F}f](\cdot) \vert^{q} \Big )^{\frac{1}{q}} \Big \vert \mathcal{M}^{u}_{p}(\R) \Big \Vert  < \infty.
    \end{align*}
    If $ q = \infty$, the usual modifications are made.
\end{defi}

Let us collect some well-known basic properties of Besov-Morrey and Triebel-Lizorkin-Morrey spaces. Most of them will be used in proofs later on. 

\begin{lem}\label{l_bp1}
    Let $ s \in \mathbb{R} $,  $ 0 < p \leq u < \infty $ and $ 0 < q \leq \infty $. Let $ \mathcal{A} \in \{ \mathcal{E}, \mathcal{N}   \}   $.  Then the following holds.
    \begin{enumerate}
        \item $\mathcal{A}^{s}_{u,p,q}(\mathbb{R}^{d})$ is independent of the chosen smooth dyadic decomposition of unity in the sense of equivalent quasi-norms. 
        
        \item The spaces $  \mathcal{A}^{s}_{u,p,q}(\mathbb{R}^{d}) $ are quasi-Banach spaces. For $ p,q \geq 1 $ they are Banach spaces.

        \item With $\tau := \min\{1,p,q\}$  we have
        $$  
            \Vert f + g  \vert \mathcal{A}^{s}_{u,p,q}(\mathbb{R}^{d}) \Vert^{\tau} \leq \Vert f \vert \mathcal{A}^{s}_{u,p,q}(\mathbb{R}^{d}) \Vert^{\tau} + \Vert g \vert \mathcal{A}^{s}_{u,p,q}(\mathbb{R}^{d}) \Vert^{\tau}, \qquad f,g \in \mathcal{A}^{s}_{u,p,q}(\mathbb{R}^{d}).
        $$
        
        \item $\mathcal{S}(\mathbb{R}^{d}) \hookrightarrow    \mathcal{A}^{s}_{u,p,q}(\mathbb{R}^{d}) \hookrightarrow   \mathcal{S}'(\mathbb{R}^{d})$.

        \item $\mathcal{E}^{s}_{p,p,q}(\mathbb{R}^{d}) = F^{s}_{p,q}(\R)$ and $\mathcal{N}^{s}_{p,p,q}(\mathbb{R}^{d}) = B^{s}_{p,q}(\R)$.

        \item If $p>1$, then $\mathcal{E}^{0}_{u,p,2}(\mathbb{R}^{d}) = \mathcal{M}^{u}_{p}(\R)$.
    \end{enumerate}
\end{lem} 

\begin{proof}
    Assertion (i) was proved in \cite[Theorem 2.8]{TangXu}. The proofs of (ii) and (iii) are standard. We refer to \cite[Lemma 2.1]{ysy} and to \cite[Corollary 2.6]{KoYa}. (iv) is proven in \cite[Proposition 2.3]{ysy} and in \cite[Theorem 3.2]{SawTan}. Finally, assertion (v) is obvious and (vi) has been shown in \cite[Proposition 4.1]{maz}.  
\end{proof}

A convenient way to show membership of distributions in $\mathcal{E}^s_{u,p,q}(\mathbb{R}^d)$ is given by the so-called \emph{dyadic ball criterion}. It has been proved by Yuan, Sickel and Yang in \cite[Proposition 6.2]{ysy}.
\begin{prop} \label{prop:dyadic_crit}
    Let $0 < p \leq u < \infty$, $0 < q \leq \infty$ and $s > \sigma_{p,q}$. 
    Further, let $(\widetilde{u}_k)_{k\in\N_0}\subset \mathcal{S}'(\R)$ satisfy 
    $$
        \supp(\mathcal{F}\widetilde{u}_k) \subseteq B(0,2^{k+2}), \qquad k \in \mathbb{N}_{0}, 
    $$
    and
    $$
        A:= \Big \Vert  \Big ( \sum_{k=0}^\infty 2^{ksq} \vert \widetilde{u}_k(\cdot) \vert^q \Big )^{\frac{1}{q}}  \Big \vert \mathcal{M}^u_p(\R) \Big \Vert < \infty
    $$
    (with the usual modification if $q=\infty$).
    Then $\sum_{k=0}^\infty \widetilde{u}_k$ converges in $\mathcal{S}'(\R)$ to some $U\in \mathcal{S}'(\R)$ and there holds
    $$
        U \in \mathcal{E}^s_{u,p,q}(\R) 
        \qquad\text{with}\qquad
        \Vert U \vert  \mathcal{E}^s_{u,p,q}(\R) \Vert \lesssim A,
    $$
    where the implicit constant does not depend on $ (\widetilde{u}_k)_{k\in\N_0}$ or $A$. 
\end{prop}

Later on, we will need embeddings of the form $ \mathcal{N}^{s_{1}}_{u_{1},p_{1},q_{1}}(\mathbb{R}^d) \hookrightarrow   \mathcal{N}^{s_{2}}_{u_{2},p_{2},q_{2}}(\mathbb{R}^d)  $ and $ \mathcal{E}^{s_{1}}_{u_{1},p_{1},q_{1}}(\mathbb{R}^d) \hookrightarrow   \mathcal{E}^{s_{2}}_{u_{2},p_{2},q_{2}}(\mathbb{R}^d)  $. Fortunately, the required embeddings already have been obtained by Haroske and Skrzypczak in \cite{HaSk2012} and \cite{HaSk2014}. For convenience of the reader we recall some of their results, but restrict ourselves to those needed below.

\begin{lem}\label{lem_FJ_3}
Let $s_{1}, s_{2} \in \mathbb{R}$, $0 < p_{1} \leq u_{1} < \infty$, $0 < p_{2} \leq u_{2} < \infty$, $ 0 < q_{1} \leq \infty$ and $ 0 < q_{2} \leq \infty$. 
Let 
\begin{align*}
u_{1} \leq u_{2} \quad \mbox{and} \quad \frac{p_{2}}{u_{2}} \leq \frac{p_{1}}{u_{1}} \quad \mbox{and} \quad s_{1} - \frac{d}{u_{1}} = s_{2} - \frac{d}{u_{2}} \quad \mbox{and} \quad q_{1} \leq q_{2} .
\end{align*}
Then it holds $ \mathcal{N}^{s_{1}}_{u_{1},p_{1},q_{1}}(\mathbb{R}^d) \hookrightarrow   \mathcal{N}^{s_{2}}_{u_{2},p_{2},q_{2}}(\mathbb{R}^d)  $.
\end{lem}

\begin{proof}
This result can be found in \cite[Eq. (2.1)]{HaSk2014}, see also \cite{HaSk2012}.
\end{proof}

\begin{lem}\label{lem_FJ_4}
Let $s_{1}, s_{2} \in \mathbb{R}$, $0 < p_{1} \leq u_{1} < \infty$, $0 < p_{2} \leq u_{2} < \infty$, $ 0 < q_{1} \leq \infty$ and $ 0 < q_{2} \leq \infty$. 
Let 
\begin{align*}
u_{1} \leq u_{2} \quad \mbox{and} \quad \frac{p_{2}}{u_{2}} \leq \frac{p_{1}}{u_{1}} \quad \mbox{and} \quad s_{1} - \frac{d}{u_{1}} = s_{2} - \frac{d}{u_{2}} \quad \mbox{and} \quad u_{1} \neq u_{2} .
\end{align*}
Then it holds $ \mathcal{E}^{s_{1}}_{u_{1},p_{1},q_{1}}(\mathbb{R}^d) \hookrightarrow   \mathcal{E}^{s_{2}}_{u_{2},p_{2},q_{2}}(\mathbb{R}^d)  $.
\end{lem}

\begin{proof}
This result is given in \cite[Theorem 2.1]{HaSk2014}.
\end{proof}

Below, we also apply some embeddings of Franke-Jawerth type. Again, we rely on the results provided by Haroske and Skrzypczak in \cite{HaSk2014}. We recall them according to their application below.  

\begin{lem}\label{lem_FJ_1}
Let $s_{1}, s_{2} \in \mathbb{R}$, $0 < p_{1} \leq u_{1} < \infty$, $0 < p_{2} \leq u_{2} < \infty$, $ 0 < q_{1} \leq \infty$ and $ 0 < q_{2} \leq \infty$. Let
\begin{align*}
p_{1} < u_{1} \quad \mbox{and} \quad \frac{p_{2}}{u_{2}} \leq \frac{p_{1}}{u_{1}} \quad \mbox{and} \quad   s_{1} - \frac{d}{u_{1}} = s_{2} - \frac{d}{u_{2}}  \quad \mbox{and} \quad s_{1} > s_{2} \quad \mbox{and} \quad   q_{2} = \infty .
\end{align*}
Then it holds $ \mathcal{E}^{s_{1}}_{u_{1},p_{1},q_{1}}(\mathbb{R}^d) \hookrightarrow   \mathcal{N}^{s_{2}}_{u_{2},p_{2},q_{2}}(\mathbb{R}^d)  $.
\end{lem}

\begin{proof}
This result can be found in \cite[Theorem 3.1]{HaSk2014}.
\end{proof}

\begin{lem}\label{lem_FJ_2}
Let $s_{1}, s_{2} \in \mathbb{R}$, $0 < p_{1} \leq u_{1} < \infty$, $0 < p_{2} \leq u_{2} < \infty$ and $ 0 < q_{2} \leq \infty$. Let
\begin{align*}
p_{1} < u_{1} \quad \mbox{and} \quad \frac{d}{p_{2}} - s_{2} = \frac{d}{p_{1}} - s_{1}  \quad \mbox{and} \quad   s_{1} - \frac{d}{u_{1}} = s_{2} - \frac{d}{u_{2}}  \quad \mbox{and} \quad s_{1} > s_{2}.
\end{align*}
Then it holds $ \mathcal{N}^{s_{1}}_{u_{1},p_{1},p_{2}}(\mathbb{R}^d) \hookrightarrow   \mathcal{E}^{s_{2}}_{u_{2},p_{2},q_{2}}(\mathbb{R}^d)  $.
\end{lem}

\begin{proof}
This result is proved in \cite[Theorem 3.2]{HaSk2014}.
\end{proof}

Let us mention that forerunners of Lemma \ref{lem_FJ_1} and Lemma \ref{lem_FJ_2} for $ u_{1} = p_{1}   $ and  $ u_{2} = p_{2}   $ can be found in \cite[Theorem 3.2.1]{SiTri}.

\section{Pointwise Multiplication and Paramultiplication}\label{sec_paramultiplication}

One main tool for the proof of Theorem \ref{thm_main_hölder1} is the technique of paramultiplication. It has been developed independently by Peetre \cite{Pee76} and Triebel \cite{Tr77} around 1976. Later on, it has been used successfully to prove various results concerning pointwise multipliers. Below, we briefly recall the main idea of paramultiplication as given in \cite[Section 2.2]{SiTri}. For $ f \in \mathcal{S}'(\R) $ and $ j \in \mathbb{N}_{0}   $ we define
\begin{equation}\label{eq_f_lowerindex}
f_{j}(x) := \mathcal{F}^{-1}[\varphi_{j} \mathcal{F}f](x)
\end{equation}
and
\begin{equation}\label{eq_def_fjoben}
f^{j}(x) := \mathcal{F}^{-1}[\varphi_{0}(2^{-j} \cdot ) \mathcal{F}f](x) ,
\end{equation}
where $ (\varphi_{j})_{j \in \N_0 }$ is a smooth dyadic decomposition of unity. As observed in \cite[Equation (2.2.1)]{SiTri} it holds
\begin{align*}
\lim_{j \rightarrow \infty} f^{j} = \lim_{j \rightarrow \infty} \mathcal{F}^{-1}[\varphi_{0}(2^{-j} \cdot ) \mathcal{F}f] = f
\end{align*}
for any $ f \in \mathcal{S}'(\R) $ with convergence in $  \mathcal{S}'(\R) $. Clearly, $ f^{j}  $ is an entire analytic function. Consequently, for all $ f, g \in \mathcal{S}'(\R) $ and all $j \in \mathbb{N}_{0}$ the product  $f^{j} \cdot g^{j}$ makes sense. As in \cite[Equation (2.2.2)]{SiTri} we define
\begin{equation}
f \cdot g := \lim_{j \rightarrow \infty} f^{j} \cdot g^{j}
\end{equation}
whenever this limit exists. Notice, that for $M \in \mathbb{N}_{0}$ we have
\begin{equation}
\sum_{j = 0}^{M} \varphi_{j}(x) = \varphi_{0}(2^{-M}x),
\end{equation}
see for example \cite[Equation (2.1.3)]{SiTri}. Consequently, we can obtain the following decomposition:
\begin{align*}
f \cdot g & = \lim_{j \rightarrow \infty} f^{j} \cdot g^{j} \\
& = \lim_{j \rightarrow \infty} \mathcal{F}^{-1}\Big [ \sum_{\ell = 0}^{j} \varphi_{\ell} \mathcal{F}f \Big ] \cdot \mathcal{F}^{-1} \Big [\sum_{k = 0}^{j} \varphi_{k}  \mathcal{F}g \Big ] \\
& = \lim_{j \rightarrow \infty} \Big ( \sum_{\ell = 0}^{j} f_{\ell} \Big ) \cdot \Big (  \sum_{k = 0}^{j} g_{k} \Big ) .
\end{align*}
As described in \cite[Section 2.2]{SiTri} this can be transformed into
\begin{equation}\label{splitting_1}
f \cdot g = \sum_{\ell = 2}^{\infty}  f^{\ell - 2}   g_{\ell} + \sum_{k = 2}^{\infty} f_{k}  g^{k-2}    + \sum_{k = 0}^{\infty} \sum_{\ell = k-1}^{k+1} f_{k} g_{\ell} .
\end{equation}
Here we put $g_{-1} := 0$. This splitting motivates the following abbreviations:
\begin{equation}\label{splitting_1_sum1}
S^{1}(f,g) := \sum_{\ell = 2}^{\infty}  f^{\ell - 2}   g_{\ell} ,
\end{equation}
\begin{equation}\label{splitting_1_sum2}
S^{2}(f,g) := \sum_{k = 2}^{\infty} f_{k}  g^{k-2}   , 
\end{equation}
\begin{equation}\label{splitting_1_sum3}
S^{3}(f,g) := \sum_{k = 0}^{\infty} \sum_{\ell = k-1}^{k+1} f_{k} g_{\ell} .   
\end{equation}
Consequently, \eqref{splitting_1} can be rewritten as
\begin{equation}\label{splitting_1_new}
f \cdot g = S^{1}(f,g) + S^{2}(f,g) + S^{3}(f,g)  .
\end{equation}
For the decomposition \eqref{splitting_1_new} we observe the following useful properties of the supports of the involved functions, see \cite[Equations (2.2.3) and (2.2.4)]{SiTri}:
\begin{equation}\label{splitting_1_supp1}
\supp \mathcal{F} (  f^{\ell - 2}   g_{\ell} ) \subset \{ \xi \in \mathbb{R}^d : 3 \cdot 2^{\ell - 3} \leq \vert \xi \vert \leq 11 \cdot 2^{\ell - 3}   \} ,
\end{equation}
\begin{equation}\label{splitting_1_supp2}
\supp \mathcal{F} \Big ( \sum_{\ell = k-1}^{k+1} f_{k} g_{\ell} \Big ) \subset \{ \xi \in \mathbb{R}^d :  \vert \xi \vert \leq 9 \cdot 2^{k - 1}   \} .
\end{equation}
Now we prove some Fourier multiplier estimates for $ S^{1}(f,g)   $  and $      S^{2}(f,g)$. 

\begin{prop}\label{prop_splitting_multipliers_simple}
Let $0 < p \leq u < \infty$, $0 < p_{1} \leq u_{1} < \infty$ and $0 < p_{2} \leq u_{2} < \infty$, such that $ \frac{1}{p} = \frac{1}{p_{1}} + \frac{1}{p_{2}}  $ and $ \frac{1}{u} = \frac{1}{u_{1}} + \frac{1}{u_{2}}  $. Let $j \in \mathbb{N}_{0}$. 
\begin{itemize}
\item[(i)] Then for $ S^{1}(f,g)   $ it holds
\begin{align*}
\Vert   \mathcal{F}^{-1} [\varphi_{j} \mathcal{F}S^{1}(f,g)  ]  \vert \mathcal{M}^{u}_{p}(\R)  \Vert  \lesssim \Big \Vert  \sup_{k \in \mathbb{N}_{0}} \vert  f^{k} \vert  \Big  \vert \mathcal{M}^{u_{1}}_{p_{1}}(\R)   \Big \Vert    \max_{\ell \in \{ -1, 0, 1, 2 \} }    \Vert  g_{j + \ell}   \vert \mathcal{M}^{u_{2}}_{p_{2}}(\R)  \Vert .
\end{align*}

\item[(ii)]  For  $ S^{2}(f,g)   $ it holds
\begin{align*}
\Vert   \mathcal{F}^{-1} [\varphi_{j} \mathcal{F}S^{2}(f,g)  ]  \vert \mathcal{M}^{u}_{p}(\R)  \Vert  \lesssim \Big \Vert  \sup_{k \in \mathbb{N}_{0}} \vert  g^{k} \vert  \Big  \vert \mathcal{M}^{u_{2}}_{p_{2}}(\R)   \Big \Vert   \max_{\ell \in \{ -1, 0, 1, 2 \} }   \Vert  f_{j + \ell}   \vert \mathcal{M}^{u_{1}}_{p_{1}}(\R)  \Vert .
\end{align*}

\end{itemize}

\end{prop}

\begin{proof}
\textit{Step 1.} At first we prove (i). Let $j \in \mathbb{N}_{0}$. We use \eqref{splitting_1_sum1} to find
\begin{align*}
\Vert   \mathcal{F}^{-1} [\varphi_{j} \mathcal{F}S^{1}(f,g)  ](\cdot)  \vert \mathcal{M}^{u}_{p}(\R)  \Vert  \leq   \Big  \Vert  \sum_{\ell = 2}^{\infty}  \vert  \mathcal{F}^{-1} [\varphi_{j} \mathcal{F}   (f^{\ell - 2}   g_{\ell})  ](\cdot)  \vert  \Big \vert \mathcal{M}^{u}_{p}(\R) \Big \Vert   .
\end{align*}
Recall, that for $k \in \mathbb{N}$ by \eqref{eq_supp_varphik} we have $   \supp \varphi_k \subset \{x\in \R : 2^{k-1}\le \vert x \vert \le 2^{k+1} \}   $. Moreover, \eqref{splitting_1_supp1} yields $ \supp \mathcal{F} (  f^{\ell - 2}   g_{\ell} ) \subset \{ \xi \in \mathbb{R}^d : 3 \cdot 2^{\ell - 3} \leq \vert \xi \vert \leq 11 \cdot 2^{\ell - 3}   \}   $. Hence, 
\begin{align*}
\Vert   \mathcal{F}^{-1} [\varphi_{j} \mathcal{F}S^{1}(f,g)  ](\cdot)  \vert \mathcal{M}^{u}_{p}(\R)  \Vert  & \leq   \Big  \Vert  \sum_{\ell = \max \{2, j-1 \} }^{j+2} \vert  \mathcal{F}^{-1} [\varphi_{j} \mathcal{F}  ( f^{\ell - 2}   g_{\ell} ) ](\cdot) \vert   \Big \vert \mathcal{M}^{u}_{p}(\R) \Big \Vert  \\
& =   \Big  \Vert  \sum_{\ell = \max \{2-j, -1 \} }^{2} \vert  \mathcal{F}^{-1} [\varphi_{j} \mathcal{F} (  f^{\ell + j - 2}   g_{\ell + j} ) ](\cdot)    \vert  \Big \vert \mathcal{M}^{u}_{p}(\R) \Big \Vert  \\
& \lesssim \sum_{\ell = \max \{2-j, -1 \} }^{2}   \Big  \Vert    \mathcal{F}^{-1} [\varphi_{j} \mathcal{F} (  f^{\ell + j - 2}   g_{\ell + j} ) ](\cdot)    \Big \vert \mathcal{M}^{u}_{p}(\R) \Big \Vert  \\
& \lesssim \max_{\ell \in \{ -1, 0, 1, 2 \} }   \Big  \Vert    \mathcal{F}^{-1} [\varphi_{j} \mathcal{F}   ( f^{\ell + j - 2}   g_{\ell + j} ) ](\cdot)    \Big \vert \mathcal{M}^{u}_{p}(\R) \Big \Vert . 
\end{align*}
For $ k < 0  $ we put $ f^k := 0   $. We use the Fourier multiplier assertion given in \cite[Theorem 2.4]{SawTan}. For that purpose let $\sigma > \frac{d}{\min(1,p)} + \frac{d}{2}$. By \eqref{splitting_1_supp1} we find that $ \supp \mathcal{F} (  f^{\ell + j - 2}   g_{\ell + j} ) \subset \{ \xi \in \mathbb{R}^d : 3 \cdot 2^{\ell + j - 3} \leq \vert \xi \vert \leq 11 \cdot 2^{\ell + j - 3}   \}$. Therefore, we obtain
\begin{align*}
& \Vert   \mathcal{F}^{-1} [\varphi_{j} \mathcal{F}S^{1}(f,g)  ](\cdot)  \vert \mathcal{M}^{u}_{p}(\R)  \Vert  \\
& \qquad   \lesssim \max_{\ell \in \{ -1, 0, 1, 2 \} }    \Vert \varphi_{j}(   11 \cdot 2^{\ell + j - 3} x )  \vert H^{\sigma}_{2}(\mathbb{R}^d)   \Vert    \Vert     f^{\ell + j - 2}   g_{\ell + j}      \vert \mathcal{M}^{u}_{p}(\R)  \Vert . 
\end{align*}
Here $  H^{\sigma}_{2}(\mathbb{R}^d)   $ is a Bessel-Potential space with the norm $  \Vert \cdot  \vert H^{\sigma}_{2}(\mathbb{R}^d)   \Vert      $.  For $j \in \mathbb{N}$ it holds  $\varphi_j( 11 \cdot 2^{\ell + j - 3} x ) = \varphi_0(  11 \cdot 2^{\ell  - 3}  x) - \varphi_0(  11 \cdot 2^{\ell  - 2} x)$. Since $\varphi_0 \in C_0^{\infty}({\R})$, we find 
\begin{align*}
\Vert   \mathcal{F}^{-1} [\varphi_{j} \mathcal{F}S^{1}(f,g)  ](\cdot)  \vert \mathcal{M}^{u}_{p}(\R)  \Vert    \lesssim \max_{\ell \in \{ -1, 0, 1, 2 \} }       \Vert     f^{\ell + j - 2}   g_{\ell + j}      \vert \mathcal{M}^{u}_{p}(\R)  \Vert . 
\end{align*}
To continue we apply the Hölder inequality with $ \frac{1}{p} = \frac{1}{p_{1}} + \frac{1}{p_{2}}  $, which implies $ 1 = \frac{p}{p_{1}} + \frac{p}{p_{2}}  $. Then we find
\begin{align*}
& \Vert   \mathcal{F}^{-1} [\varphi_{j} \mathcal{F}S^{1}(f,g)  ](\cdot)  \vert \mathcal{M}^{u}_{p}(\R)  \Vert  \\
&   \lesssim \max_{\ell \in \{ -1, 0, 1, 2 \} }  \sup_{y \in \R, r > 0} \abs{B(y,r)}^{\frac{1}{u}-\frac{1}{p}} \Big ( \int_{B(y,r)}  \sup_{k \in \mathbb{N}_{0}} \vert  f^{k}(x) \vert^{p}    \vert       g_{\ell + j}(x)  \vert^{p}  dx \Big )^{\frac{1}{p}}  \\
&  \lesssim  \max_{\ell \in \{ -1, 0, 1, 2 \} }  \sup_{y \in \R, r > 0} \abs{B(y,r)}^{\frac{1}{u}-\frac{1}{p}}   \Big ( \int_{B(y,r)}  \sup_{k \in \mathbb{N}_{0}} \vert  f^{k}(x) \vert^{p_{1}} dx \Big )^{\frac{1}{p_{1}}} \Big (  \int_{B(y,r)}    \vert    g_{\ell + j}(x)  \vert^{p_{2}}  dx \Big )^{\frac{1}{p_{2}}}     .
\end{align*}
Recall $ \frac{1}{u} = \frac{1}{u_{1}} + \frac{1}{u_{2}}  $. Therefore
\begin{equation}\label{eq_u_ballsplit_u1u2}
\abs{B(y,r)}^{\frac{1}{u}-\frac{1}{p}}  =  \abs{B(y,r)}^{\frac{1}{u_{1}} + \frac{1}{u_{2}}-  \frac{1}{p_{1}} - \frac{1}{p_{2}}  } = \abs{B(y,r)}^{\frac{1}{u_{1}}-\frac{1}{p_{1}}}  \abs{B(y,r)}^{\frac{1}{u_{2}}-\frac{1}{p_{2}}}  .  
\end{equation}
Consequently,
\begin{align*}
 \Vert   \mathcal{F}^{-1} [\varphi_{j} \mathcal{F}S^{1}(f,g)  ]  \vert \mathcal{M}^{u}_{p}(\R)  \Vert   \lesssim  \Big \Vert  \sup_{k \in \mathbb{N}_{0}} \vert  f^{k} \vert   \Big \vert \mathcal{M}^{u_{1}}_{p_{1}}(\R) \Big \Vert   \max_{\ell \in \{ -1, 0, 1, 2 \} }   \Big \Vert          g_{\ell + j}      \Big \vert \mathcal{M}^{u_{2}}_{p_{2}}(\R) \Big \Vert    .
\end{align*}
This completes Step 1.

\textit{Step 2.} Now we prove (ii). For that purpose we use \eqref{splitting_1_sum2} to find  
\begin{equation}\label{eq_S2=S1}
S^{2}(f,g) = \sum_{k = 2}^{\infty} f_{k}  g^{k-2}   =  \sum_{k = 2}^{\infty}    g^{k-2} f_{k} = S^{1}(g,f) . 
\end{equation}
Hence, we can proceed as in Step 1 to obtain the desired result. 
\end{proof}

To continue we prove an advanced version of Proposition \ref{prop_splitting_multipliers_simple} for function series and deal with $   S^{3}(f,g) $. The following result can be seen as a generalization of \cite[Proposition 2.2.1]{SiTri} to Morrey spaces.

\begin{prop}\label{prop_splitting_multipliers}
Let $0 < p \leq u < \infty$, $0 < p_{1} \leq u_{1} < \infty$ and $0 < p_{2} \leq u_{2} < \infty$, such that $ \frac{1}{p} = \frac{1}{p_{1}} + \frac{1}{p_{2}}  $ and $ \frac{1}{u} = \frac{1}{u_{1}} + \frac{1}{u_{2}}  $. Let $ 0 < q \leq \infty   $.
\begin{itemize}
\item[(i)]  It holds
\begin{align*}
& \Big \Vert \Big ( \sum_{k = 0}^{\infty} \Big  \vert \mathcal{F}^{-1}\Big [\varphi_{k} \mathcal{F}S^{1}(f,g) \Big ](\cdot) \Big \vert^{q} \Big )^{\frac{1}{q}} \Big \vert \mathcal{M}^{u}_{p}(\R) \Big \Vert \\
& \qquad \qquad \lesssim \Big \Vert  \sup_{k \in \mathbb{N}_{0}} \vert  f^{k}( \cdot ) \vert   \Big \vert \mathcal{M}^{u_{2}}_{p_{2}}(\R) \Big \Vert \Big \Vert \Big ( \sum_{k = 0}^{\infty}   \vert g_{k}(\cdot) \vert^{q} \Big )^{\frac{1}{q}} \Big \vert \mathcal{M}^{u_{1}}_{p_{1}}(\R) \Big \Vert .
\end{align*}
In the case $q = \infty$ the usual modifications are made.
\item[(ii)]  Assume $p > 1$ and $k \in \mathbb{N}_{0}$. Then it holds
\begin{align*}
 \Big \Vert    \mathcal{F}^{-1}\Big [\varphi_{k} \mathcal{F}S^{3}(f,g) \Big ]  \Big \vert \mathcal{M}^{u}_{p}(\R) \Big \Vert  \lesssim \max_{-1 \leq j \leq 1} \sum_{\ell = -2}^{\infty} \Vert f_{k+ \ell} \vert \mathcal{M}^{u_{1}}_{p_{1}} (\mathbb{R}^d) \Vert  \Vert g_{k+ \ell + j} \vert \mathcal{M}^{u_{2}}_{p_{2}} (\mathbb{R}^d) \Vert .
\end{align*}
We put $f_{r} = g_{r} = 0$ if $r < 0$.
\item[(iii)]  Assume $p \leq 1$ and $k \in \mathbb{N}_{0}$. Then it holds
\begin{align*}
& \Big \Vert    \mathcal{F}^{-1}\Big [\varphi_{k} \mathcal{F}S^{3}(f,g) \Big ](\cdot)  \Big \vert \mathcal{M}^{u}_{p}(\R) \Big \Vert \\
& \qquad \qquad \lesssim \max_{-1 \leq j \leq 1} \Big ( \sum_{\ell = -2}^{\infty} 2^{\ell d(1-p)} \Vert f_{k+ \ell} \vert \mathcal{M}^{u_{1}}_{p_{1}} (\mathbb{R}^d) \Vert^{p}  \Vert g_{k+ \ell + j} \vert \mathcal{M}^{u_{2}}_{p_{2}} (\mathbb{R}^d) \Vert^{p}  \Big )^{\frac{1}{p}} .
\end{align*}
Again we put $f_{r} = g_{r} = 0$ if $r < 0$.

\item[(iv)] Let $s > \sigma_{p}$. Then it holds
\begin{align*}
& \Big \Vert  \sup_{k \in \mathbb{N}_{0}}  2^{ks} \Big \vert \mathcal{F}^{-1}\Big [\varphi_{k} \mathcal{F}S^{3}(f,g) \Big ](\cdot) \Big \vert   \Big \vert \mathcal{M}^{u}_{p}(\R) \Big \Vert \\
& \qquad \qquad \lesssim \max_{-1 \leq j \leq 1}  \Big \Vert  \sup_{k \in \mathbb{N}_{0}} 2^{\frac{ks}{2}} \vert f_{k} \vert  \Big \vert \mathcal{M}^{u_{1}}_{p_{1}} (\mathbb{R}^d) \Big \Vert \Big  \Vert   \sup_{k \in \mathbb{N}_{0}} 2^{\frac{ks}{2}} \vert g_{k + j} \vert  \Big \vert \mathcal{M}^{u_{2}}_{p_{2}} (\mathbb{R}^d) \Big  \Vert .
\end{align*}

\end{itemize}
\end{prop}

\begin{proof}
\textit{Step 1.} At first we prove (i). For that purpose we use \eqref{splitting_1_sum1} to find
\begin{align*}
& \Big \Vert \Big ( \sum_{k = 0}^{\infty} \Big  \vert \mathcal{F}^{-1}\Big [\varphi_{k} \mathcal{F}S^{1}(f,g) \Big ](\cdot) \Big \vert^{q} \Big )^{\frac{1}{q}} \Big \vert \mathcal{M}^{u}_{p}(\R) \Big \Vert \\
& \qquad \qquad =  \Big \Vert \Big ( \sum_{k = 0}^{\infty} \Big  \vert \mathcal{F}^{-1}\Big [\varphi_{k} \mathcal{F} \Big (  \sum_{\ell = 2}^{\infty}  f^{\ell - 2}   g_{\ell} \Big )  \Big ](\cdot) \Big \vert^{q} \Big )^{\frac{1}{q}} \Big \vert \mathcal{M}^{u}_{p}(\R) \Big \Vert   \\
& \qquad \qquad \leq \Big \Vert \Big (  \sum_{\ell = 2}^{\infty} \sum_{k = 0}^{\infty}   \vert \mathcal{F}^{-1} [\varphi_{k} \mathcal{F} (    f^{\ell - 2}   g_{\ell}  )   ](\cdot)  \vert^{q} \Big )^{\frac{1}{q}} \Big \vert \mathcal{M}^{u}_{p}(\R) \Big \Vert   .
\end{align*}
For $k \in \mathbb{N}$ by \eqref{eq_supp_varphik} it follows $   \supp \varphi_k \subset \{x\in \R : 2^{k-1}\le \vert x \vert \le 2^{k+1} \}   $. \eqref{splitting_1_supp1} implies $ \supp \mathcal{F} (  f^{\ell - 2}   g_{\ell} ) \subset \{ \xi \in \mathbb{R}^d : 3 \cdot 2^{\ell - 3} \leq \vert \xi \vert \leq 11 \cdot 2^{\ell - 3}   \}  \subset \{ \xi \in \mathbb{R}^d : 2 \cdot 2^{\ell - 3} \leq \vert \xi \vert \leq 2^4 \cdot 2^{\ell - 3}   \}   $. Consequently, the support properties of the involved functions imply
\begin{align*}
& \Big \Vert \Big ( \sum_{k = 0}^{\infty} \Big  \vert \mathcal{F}^{-1}\Big [\varphi_{k} \mathcal{F}S^{1}(f,g) \Big ](\cdot) \Big \vert^{q} \Big )^{\frac{1}{q}} \Big \vert \mathcal{M}^{u}_{p}(\R) \Big \Vert \\
& \qquad \qquad \leq \Big \Vert \Big (  \sum_{\ell = 2}^{\infty} \sum_{n = -2}^{1}   \vert \mathcal{F}^{-1} [\varphi_{\ell + n} \mathcal{F} (    f^{\ell - 2}   g_{\ell}  )   ](\cdot)  \vert^{q} \Big )^{\frac{1}{q}} \Big \vert \mathcal{M}^{u}_{p}(\R) \Big \Vert   \\
& \qquad \qquad \lesssim  \max_{n \in \{-2, -1, 0, 1\} } \Big \Vert \Big (  \sum_{\ell = 2}^{\infty}    \vert \mathcal{F}^{-1} [\varphi_{\ell + n} \mathcal{F} (    f^{\ell - 2}   g_{\ell}  )   ](\cdot)  \vert^{q} \Big )^{\frac{1}{q}} \Big \vert \mathcal{M}^{u}_{p}(\R) \Big \Vert   .
\end{align*}
We use a Fourier multiplier assertion given in \cite[Theorem 2.4]{SawTan}. For that purpose let $\sigma > \frac{d}{\min(1,p,q)} + \frac{d}{2}$. Since $ \supp \mathcal{F} (  f^{\ell - 2}   g_{\ell} ) \subset \{ \xi \in \mathbb{R}^d : 3 \cdot 2^{\ell - 3} \leq \vert \xi \vert \leq 11 \cdot 2^{\ell - 3}   \}$, we obtain
\begin{align*}
& \Big \Vert \Big ( \sum_{k = 0}^{\infty} \Big  \vert \mathcal{F}^{-1}\Big [\varphi_{k} \mathcal{F}S^{1}(f,g) \Big ](\cdot) \Big \vert^{q} \Big )^{\frac{1}{q}} \Big \vert \mathcal{M}^{u}_{p}(\R) \Big \Vert \\
& \qquad \qquad \lesssim  \max_{n \in \{-2, -1, 0, 1\} } \sup_{k \in \mathbb{N}_{0}}  \Vert  \varphi_{k + n}(11 \cdot 2^{k-3}  x ) \vert   H^{\sigma}_{2}(\mathbb{R}^d) \Vert \Big \Vert \Big (  \sum_{\ell = 2}^{\infty}    \vert     f^{\ell - 2}   g_{\ell}     \vert^{q} \Big )^{\frac{1}{q}} \Big \vert \mathcal{M}^{u}_{p}(\R) \Big \Vert   .
\end{align*}
Here for $r < 0$ we put $  \varphi_{r} = 0  $. Recall, that for $k + n \in \mathbb{N}$ it holds
\begin{align*}
 \varphi_{k+n}(11 \cdot 2^{k-3} \cdot x) & = \varphi_0( 11 \cdot 2^{k-3} \cdot  2^{-k-n}x) - \varphi_0( 11 \cdot 2^{k-3} \cdot  2^{-k-n+1}x) \\
& = \varphi_0( 11  \cdot  2^{-n-3}x) - \varphi_0( 11  \cdot  2^{-n-2}x) . 
\end{align*}
Due to $ n \in \{-2, -1, 0, 1\}    $ and $\varphi_0 \in C_0^{\infty}({\R})$ this implies
\begin{align*}
\Big \Vert \Big ( \sum_{k = 0}^{\infty} \Big  \vert \mathcal{F}^{-1}\Big [\varphi_{k} \mathcal{F}S^{1}(f,g) \Big ](\cdot) \Big \vert^{q} \Big )^{\frac{1}{q}} \Big \vert \mathcal{M}^{u}_{p}(\R) \Big \Vert  \lesssim  \Big \Vert \Big (  \sum_{\ell = 2}^{\infty}    \vert     f^{\ell - 2}   g_{\ell}     \vert^{q} \Big )^{\frac{1}{q}} \Big \vert \mathcal{M}^{u}_{p}(\R) \Big \Vert   .
\end{align*}
We apply the Hölder inequality with $ \frac{1}{p} = \frac{1}{p_{1}} + \frac{1}{p_{2}}  $, which implies $ 1 = \frac{p}{p_{1}} + \frac{p}{p_{2}}  $, to find
\begin{align*}
& \Big \Vert \Big ( \sum_{k = 0}^{\infty} \Big  \vert \mathcal{F}^{-1}\Big [\varphi_{k} \mathcal{F}S^{1}(f,g) \Big ](\cdot) \Big \vert^{q} \Big )^{\frac{1}{q}} \Big \vert \mathcal{M}^{u}_{p}(\R) \Big \Vert \\
& \quad  \lesssim  \sup_{y \in \R, r > 0} \abs{B(y,r)}^{\frac{1}{u}-\frac{1}{p}} \Big ( \int_{B(y,r)}  \sup_{k \in \mathbb{N}_{0}} \vert  f^{k}(x) \vert^{p}  \Big (  \sum_{\ell = 0}^{\infty}    \vert       g_{\ell}(x)     \vert^{q} \Big )^{\frac{p}{q}}  dx \Big )^{\frac{1}{p}}    \\
& \quad  \lesssim  \sup_{y \in \R, r > 0} \abs{B(y,r)}^{\frac{1}{u}-\frac{1}{p}}   \Big ( \int_{B(y,r)}  \sup_{k \in \mathbb{N}_{0}} \vert  f^{k}(x) \vert^{p_{2}} dx \Big )^{\frac{1}{p_{2}}} \Big (  \int_{B(y,r)}  \Big (  \sum_{\ell = 0}^{\infty}    \vert       g_{\ell}(x)     \vert^{q} \Big )^{\frac{p_{1}}{q}}  dx \Big )^{\frac{1}{p_{1}}}     .
\end{align*}
Recall $ \frac{1}{u} = \frac{1}{u_{1}} + \frac{1}{u_{2}}  $ and \eqref{eq_u_ballsplit_u1u2}. Consequently,
\begin{align*}
& \Big \Vert \Big ( \sum_{k = 0}^{\infty} \Big  \vert \mathcal{F}^{-1}\Big [\varphi_{k} \mathcal{F}S^{1}(f,g) \Big ](\cdot) \Big \vert^{q} \Big )^{\frac{1}{q}} \Big \vert \mathcal{M}^{u}_{p}(\R) \Big \Vert \\
& \qquad  \lesssim  \Big \Vert  \sup_{k \in \mathbb{N}_{0}} \vert  f^{k}( \cdot ) \vert   \Big \vert \mathcal{M}^{u_{2}}_{p_{2}}(\R) \Big \Vert \Big \Vert  \Big (  \sum_{\ell = 0}^{\infty}    \vert       g_{\ell}( \cdot )     \vert^{q} \Big )^{\frac{1}{q}}   \Big \vert \mathcal{M}^{u_{1}}_{p_{1}}(\R) \Big \Vert    .
\end{align*}
This completes Step 1.

\textit{Step 2.} Now we prove (ii). For that purpose we use the support properties of the involved functions, namely \eqref{eq_supp_varphik} and \eqref{splitting_1_supp2}. Using \eqref{splitting_1_sum3}, for all $k \in \mathbb{N}_{0}$ they imply
\begin{align*}
 \mathcal{F}^{-1}\Big [\varphi_{k} \mathcal{F}S^{3}(f,g) \Big ] & = \mathcal{F}^{-1}\Big [\varphi_{k} \mathcal{F} \Big ( \sum_{\ell = 0}^{\infty} \sum_{j = \ell-1}^{\ell+1} f_{\ell} g_{j} \Big )  \Big ] \\
& = \mathcal{F}^{-1}\Big [\varphi_{k} \mathcal{F} \Big ( \sum_{\ell = 0}^{\infty} \sum_{j = -1}^{1} f_{\ell} g_{\ell + j} \Big )  \Big ] \\ 
& =  \sum_{t = -2}^{\infty} \sum_{j = -1}^{1} \mathcal{F}^{-1} [\varphi_{k} \mathcal{F} (  f_{k + t} \cdot g_{k + t + j} )  ] , 
\end{align*}
see also \cite[Equation (5.5.15)]{SiTri}. To continue we apply a Fourier multiplier theorem of Mikhlin type for Morrey spaces, see \cite[Theorem 3]{MaWietall}. Let $\alpha \in \mathbb{N}_{0}^{d}$ be a multi-index with $\vert \alpha \vert \leq d + 2$. Then for $k \in \mathbb{N}$ we observe
\begin{align*}
\vert D^{\alpha} \varphi_{k}(x) \vert & = \vert D^{\alpha} \varphi_0(2^{-k}x) - D^{\alpha} \varphi_0(2^{-k+1}x) \vert \\
& \leq \vert 2^{-k \vert \alpha \vert} (D^{\alpha} \varphi_0)(2^{-k}x) \vert  + \vert  2^{(-k+1) \vert \alpha \vert }  (D^{\alpha} \varphi_0)(2^{-k+1}x) \vert \\
& \leq  C_{\alpha} \chi_{B(0, \frac{3}{2} 2^{k})}(x) 2^{-k \vert \alpha \vert}    +    C_{\alpha} \chi_{B(0, \frac{3}{2} 2^{k-1})}(x) 2^{(-k+1) \vert \alpha \vert }   \\
& \leq  C'_{\alpha} \vert x \vert^{- \vert \alpha \vert} .
\end{align*}
In the last step we used the compact support of the indicator function $ \chi_{B(0, \frac{3}{2} 2^{k})}( \cdot )  $. The constant $  C'_{\alpha}  $ is independent of $x$ and $k$. Consequently, the Fourier multiplier theorem can be applied. It requires $p > 1$ and yields
\begin{align*}
& \Big \Vert    \mathcal{F}^{-1}\Big [\varphi_{k} \mathcal{F}S^{3}(f,g) \Big ](\cdot)  \Big \vert \mathcal{M}^{u}_{p}(\R) \Big \Vert \\
& \qquad \qquad = \Big \Vert   \sum_{t = -2}^{\infty} \sum_{j = -1}^{1} \mathcal{F}^{-1} [\varphi_{k} \mathcal{F} (  f_{k + t} \cdot g_{k + t + j} )  ] (\cdot)  \Big \vert \mathcal{M}^{u}_{p}(\R) \Big \Vert \\
& \qquad \qquad \leq    \sum_{t = -2}^{\infty} \sum_{j = -1}^{1} \Vert \mathcal{F}^{-1} [\varphi_{k} \mathcal{F} (  f_{k + t} \cdot g_{k + t + j} )  ] (\cdot)   \vert \mathcal{M}^{u}_{p}(\R)  \Vert \\
& \qquad \qquad \lesssim    \sum_{t = -2}^{\infty} \sum_{j = -1}^{1} \Vert    f_{k + t} \cdot g_{k + t + j}      \vert \mathcal{M}^{u}_{p}(\R)  \Vert \\
& \qquad \qquad \lesssim \max_{j \in \{ -1, 0, 1  \}}    \sum_{t = -2}^{\infty}  \Vert    f_{k + t} \cdot g_{k + t + j}      \vert \mathcal{M}^{u}_{p}(\R)  \Vert .
\end{align*}
A combination of the Hölder inequality with $ 1 = \frac{p}{p_{1}} + \frac{p}{p_{2}}  $ and \eqref{eq_u_ballsplit_u1u2} yields 
\begin{align*}
& \Big \Vert    \mathcal{F}^{-1}\Big [\varphi_{k} \mathcal{F}S^{3}(f,g) \Big ](\cdot)  \Big \vert \mathcal{M}^{u}_{p}(\R) \Big \Vert \\
& \qquad \qquad \lesssim \max_{j \in \{ -1, 0, 1  \}}    \sum_{t = -2}^{\infty}  \Vert    f_{k + t}  \vert \mathcal{M}^{u_{1}}_{p_{1}}(\R)  \Vert  \Vert g_{k + t + j}      \vert \mathcal{M}^{u_{2}}_{p_{2}}(\R)  \Vert .
\end{align*}
Hence, the proof of (ii) is complete. 

\textit{Step 3.} To continue we prove (iii). Let $ 0 < p \leq 1$. As in Step 2 we find
\begin{align*}
& \Big \Vert    \mathcal{F}^{-1}\Big [\varphi_{k} \mathcal{F}S^{3}(f,g) \Big ](\cdot)  \Big \vert \mathcal{M}^{u}_{p}(\R) \Big \Vert^{p} \\
& \qquad \qquad = \Big \Vert   \sum_{t = -2}^{\infty} \sum_{j = -1}^{1} \mathcal{F}^{-1} [\varphi_{k} \mathcal{F} (  f_{k + t} \cdot g_{k + t + j} )  ] (\cdot)  \Big \vert \mathcal{M}^{u}_{p}(\R) \Big \Vert^{p} \\
& \qquad \qquad \leq    \sum_{t = -2}^{\infty} \sum_{j = -1}^{1} \Vert \mathcal{F}^{-1} [\varphi_{k} \mathcal{F} (  f_{k + t} \cdot g_{k + t + j} )  ] (\cdot)   \vert \mathcal{M}^{u}_{p}(\R)  \Vert^{p}  \\
& \qquad \qquad \lesssim \max_{j \in \{ -1, 0, 1 \} }    \sum_{t = -2}^{\infty}  \Vert \mathcal{F}^{-1} [\varphi_{k} \mathcal{F} (  f_{k + t} \cdot g_{k + t + j} )  ] (\cdot)   \vert \mathcal{M}^{u}_{p}(\R)  \Vert^{p}  .
\end{align*}
Using \cite[Chapter 1.5.2, Remark 3]{Tr83} and \cite[Theorem 2.4]{HoWeiDil} we obtain
\begin{align*}
& \Vert \mathcal{F}^{-1} [\varphi_{k} \mathcal{F} (  f_{k + t} \cdot g_{k + t + j} )  ] (\cdot)   \vert \mathcal{M}^{u}_{p}(\R)  \Vert \\
& \qquad =   \Vert ( \mathcal{F}^{-1} [\varphi_{k}(2^{k+t+2} \cdot ) (\mathcal{F} (  f_{k + t} \cdot g_{k + t + j} )(2^{-k-t-2} \cdot))(\cdot)  ]) (2^{k+t+2} \cdot)   \vert \mathcal{M}^{u}_{p}(\R)  \Vert \\
& \qquad \lesssim 2^{ - \frac{d}{u}(k+t+2)}   \Vert  \mathcal{F}^{-1} [\varphi_{k}(2^{k+t+2} \cdot ) (\mathcal{F} (  f_{k + t} \cdot g_{k + t + j} )(2^{-k-t-2} \cdot))(\cdot)  ] \vert \mathcal{M}^{u}_{p}(\R)  \Vert \\
& \qquad \lesssim (2^{-t-1} + 1)^{d(\frac{1}{p}-1)}   \Vert  \mathcal{F}^{-1} [\varphi_{k}(2^{k+t+2} \cdot )] \vert  L_{p}(\mathbb{R}^d) \Vert  \Vert     f_{k + t} \cdot g_{k + t + j}   \vert \mathcal{M}^{u}_{p}(\R)  \Vert .
\end{align*}
The equation $ \varphi_{k}(2^{k+t+2} x ) = \varphi_{0}(2^{t+2}x) - \varphi_{0}(2^{t+3}x)   $ yields
\begin{align*}
& \Vert \mathcal{F}^{-1} [\varphi_{k} \mathcal{F} (  f_{k + t} \cdot g_{k + t + j} )  ] (\cdot)   \vert \mathcal{M}^{u}_{p}(\R)  \Vert  \lesssim  2^{td(\frac{1}{p}-1)}  \Vert     f_{k + t} \cdot g_{k + t + j}   \vert \mathcal{M}^{u}_{p}(\R)  \Vert .
\end{align*}
Using this in combination with the Hölder inequality with $ 1 = \frac{p}{p_{1}} + \frac{p}{p_{2}}  $ and \eqref{eq_u_ballsplit_u1u2} we find
\begin{align*}
& \Big \Vert    \mathcal{F}^{-1}\Big [\varphi_{k} \mathcal{F}S^{3}(f,g) \Big ](\cdot)  \Big \vert \mathcal{M}^{u}_{p}(\R) \Big \Vert^{p} \\
& \qquad \qquad \lesssim \max_{j \in \{ -1, 0, 1 \} }    \sum_{t = -2}^{\infty} 2^{td(1-p)}  \Vert     f_{k + t} \cdot g_{k + t + j}   \vert \mathcal{M}^{u}_{p}(\R)  \Vert^{p}   \\
& \qquad \qquad \lesssim \max_{j \in \{ -1, 0, 1 \} }    \sum_{t = -2}^{\infty} 2^{td(1-p)}  \Vert    f_{k + t}  \vert \mathcal{M}^{u_{1}}_{p_{1}}(\R)  \Vert^{p}  \Vert g_{k + t + j}      \vert \mathcal{M}^{u_{2}}_{p_{2}}(\R)  \Vert^{p}    .
\end{align*}
Hence, this step of the proof is complete.

\textit{Step 4.} To continue we prove (iv).  Using \eqref{splitting_1_sum3} we find
\begin{align*}
& \Big \Vert  \sup_{k \in \mathbb{N}_{0}}  2^{ks} \Big \vert \mathcal{F}^{-1}\Big [\varphi_{k} \mathcal{F}S^{3}(f,g) \Big ](\cdot) \Big \vert   \Big \vert \mathcal{M}^{u}_{p}(\R) \Big \Vert \\
& \qquad =  \Big \Vert  \sup_{k \in \mathbb{N}_{0}}  2^{ks} \Big \vert \mathcal{F}^{-1}\Big [\varphi_{k} \mathcal{F} \Big ( \sum_{\ell = 0}^{\infty} \sum_{j = -1}^{1} f_{\ell} g_{\ell + j} \Big )  \Big ] \Big \vert   \Big \vert \mathcal{M}^{u}_{p}(\R) \Big \Vert \\
& \qquad \lesssim \sum_{j = -1}^{1} \Big \Vert  \sup_{k \in \mathbb{N}_{0}}  2^{ks} \Big \vert \mathcal{F}^{-1}\Big [\varphi_{k} \mathcal{F} \Big ( \sum_{\ell = 0}^{\infty}  f_{\ell} g_{\ell + j}  \Big ) \Big ] \Big \vert  \Big \vert \mathcal{M}^{u}_{p}(\R) \Big \Vert \\
& \qquad \lesssim \max_{j \in \{ -1, 0, 1 \} } \Big \Vert  \sup_{k \in \mathbb{N}_{0}}  2^{ks} \Big \vert \mathcal{F}^{-1}\Big [\varphi_{k} \mathcal{F} \Big ( \sum_{\ell = 0}^{\infty}  f_{\ell} g_{\ell + j} \Big ) \Big ] \Big \vert   \Big \vert \mathcal{M}^{u}_{p}(\R) \Big \Vert .
\end{align*}
Now we apply Proposition \ref{prop:dyadic_crit} with $ q = \infty  $ and $\tilde{u}_{\ell }:= f_{\ell} g_{\ell + j}$ for $\ell \in \mathbb{N}_{0}$. This is possible since $ s > \sigma_{p}$ and $\supp ( \mathcal{F} ( f_{\ell} g_{\ell + j}) ) \subset B(0, 2^{\ell + 3})$, see \eqref{splitting_1_supp2}. It yields 
\begin{align*}
\Big \Vert  \sup_{k \in \mathbb{N}_{0}}  2^{ks} \Big \vert \mathcal{F}^{-1}\Big [\varphi_{k} \mathcal{F}S^{3}(f,g) \Big ](\cdot) \Big \vert   \Big \vert \mathcal{M}^{u}_{p}(\R) \Big \Vert  \lesssim \max_{j \in \{ -1, 0, 1 \} } \Big \Vert  \sup_{k \in \mathbb{N}_{0}}  2^{ks} \vert   f_{k} g_{k + j}    \vert   \Big \vert \mathcal{M}^{u}_{p}(\R) \Big \Vert .
\end{align*}
The Hölder inequality with $ 1 = \frac{p}{p_{1}} + \frac{p}{p_{2}}  $ and \eqref{eq_u_ballsplit_u1u2} implies 
\begin{align*}
& \Big \Vert  \sup_{k \in \mathbb{N}_{0}}  2^{ks} \Big \vert \mathcal{F}^{-1}\Big [\varphi_{k} \mathcal{F}S^{3}(f,g) \Big ](\cdot) \Big \vert   \Big \vert \mathcal{M}^{u}_{p}(\R) \Big \Vert \\
& \qquad \qquad \lesssim \max_{-1 \leq j \leq 1}  \Big \Vert  \sup_{k \in \mathbb{N}_{0}} 2^{\frac{ks}{2}} \vert f_{k} \vert  \Big \vert \mathcal{M}^{u_{1}}_{p_{1}} (\mathbb{R}^d) \Big \Vert \Big  \Vert   \sup_{k \in \mathbb{N}_{0}} 2^{\frac{ks}{2}} \vert g_{k + j} \vert  \Big \vert \mathcal{M}^{u_{2}}_{p_{2}} (\mathbb{R}^d) \Big  \Vert .
\end{align*}
The proof is complete. 
\end{proof}
In the estimates provided by
Proposition \ref{prop_splitting_multipliers_simple} and  Proposition \ref{prop_splitting_multipliers} sometimes expressions with a supremum inside the Morrey norm show up. They can be simplified by using the following upper estimate.  

\begin{lem}\label{lem_mor_eq_hardy}
Let $1 < p \leq u < \infty$. Then for $ f \in  \mathcal{M}^{u}_{p}(\R)   $ it holds
\begin{align*}
\Big \Vert  \sup_{k \in \mathbb{N}_{0}} \vert  f^{k}(x) \vert   \Big \vert \mathcal{M}^{u}_{p}(\R) \Big \Vert \lesssim \Vert f \vert  \mathcal{M}^{u}_{p}(\R)  \Vert  .
\end{align*}
\end{lem}

\begin{proof}
For the proof at first we apply \eqref{eq_def_fjoben}. Due to $\varphi_0 \in C_0^{\infty}(\R) \subset \mathcal{S}(\R)$ and $ f \in  \mathcal{M}^{u}_{p}(\R) \subset \mathcal{S}'(\mathbb{R}^{d})  $ we can use the convolution theorem to find
\begin{align*}
\Big \Vert  \sup_{k \in \mathbb{N}_{0}} \vert  f^{k}(x) \vert   \Big \vert \mathcal{M}^{u}_{p}(\R) \Big \Vert & = \Big \Vert  \sup_{k \in \mathbb{N}_{0}} \vert  \mathcal{F}^{-1}[\varphi_{0}(2^{-k} \cdot ) \mathcal{F}f](x) \vert   \Big \vert \mathcal{M}^{u}_{p}(\R) \Big \Vert    \\
& = (2 \pi)^{ - \frac{d}{2}} \Big \Vert  \sup_{k \in \mathbb{N}_{0}} 2^{kd}  \vert  [(\mathcal{F}^{-1}\varphi_{0})(2^{k} \cdot ) \ast f](x) \vert   \Big \vert \mathcal{M}^{u}_{p}(\R) \Big \Vert    \\
& \leq (2 \pi)^{ - \frac{d}{2}} \Big \Vert  \sup_{t > 0} t^{-d}  \vert  [(\mathcal{F}^{-1}\varphi_{0})(t^{-1} \cdot ) \ast f](x) \vert   \Big \vert \mathcal{M}^{u}_{p}(\R) \Big \Vert  .
\end{align*}
Since $ \mathcal{F}^{-1}\varphi_{0} \in \mathcal{S}(\R)   $, this expression coincides with the definition of the Hardy-Morrey spaces, see for example \cite[Definition 1.3]{HouHen}. It is well known that for $1 < p \leq u < \infty$ the Hardy-Morrey spaces are equal to the Morrey spaces $ \mathcal{M}^{u}_{p}(\R)   $ with equivalent norms, see for instance the remark below \cite[Definition 1.3]{HouHen}. Some proof details concerning this observation can be found in \cite[Section 2]{HouHen}. Consequently, 
\begin{align*}
\Big \Vert  \sup_{t > 0} t^{-d}  \vert  [(\mathcal{F}^{-1}\varphi_{0})(t^{-1} \cdot ) \ast f](x) \vert   \Big \vert \mathcal{M}^{u}_{p}(\R) \Big \Vert  \lesssim  \Vert f \vert  \mathcal{M}^{u}_{p}(\R)  \Vert  .
\end{align*}
The proof is complete.
\end{proof}

\section{The Proof of Theorem \ref{thm_main_hölder1}}\label{sec_proof_mainresult}

Now we are well-prepared to prove the main result of this paper, namely Theorem \ref{thm_main_hölder1}. For that purpose we combine Proposition \ref{prop_splitting_multipliers_simple} and  Proposition \ref{prop_splitting_multipliers} with some embeddings of Franke-Jawerth type. The proof reads as follows.

\vspace{0,3 cm}

\textbf{Proof of Theorem \ref{thm_main_hölder1}.}

For the proof let $s > 0$, $0 < p_{1} \leq u_{1} < \infty$, $0 < p_{2} \leq u_{2} < \infty $,  $0 < p \leq u < \infty $, $ 0 < q_{1} \leq \infty    $, $ 0 < q_{2} \leq \infty    $  and  $ 0 < q \leq \infty    $. In addition, let \eqref{eq_hölder_meet_franke1} and \eqref{eq_hölder_meet_franke2} be fulfilled. Recall, that for $  p_{1} = u_{1}  $, $ p_{2} = u_{2}   $ and $  p = u  $ the theorem is well-known, see Theorem \ref{thm_hölder_historical1}.

\textit{Step 1.} At first we deal with the case of Besov-Morrey spaces and prove (i). Let  $ f \in  \mathcal{N}^{s}_{u_{1},p_{1},q_{1}}(\mathbb{R}^d)  $ and  $   g \in  \mathcal{N}^{s}_{u_{2},p_{2},q_{2}}(\mathbb{R}^d) $.   Using \eqref{splitting_1} - \eqref{splitting_1_new} we find
\begin{align*}
\Vert f \cdot g \vert  \mathcal{N}^{s}_{u,p,q}(\mathbb{R}^d)   \Vert & =   \Vert  S^{1}(f,g) + S^{2}(f,g) + S^{3}(f,g) \vert  \mathcal{N}^{s}_{u,p,q}(\mathbb{R}^d)   \Vert \\
& \lesssim  \Vert  S^{1}(f,g) \vert  \mathcal{N}^{s}_{u,p,q}(\mathbb{R}^d)   \Vert + \Vert S^{2}(f,g) \vert  \mathcal{N}^{s}_{u,p,q}(\mathbb{R}^d)   \Vert + \Vert S^{3}(f,g) \vert  \mathcal{N}^{s}_{u,p,q}(\mathbb{R}^d)   \Vert .
\end{align*}

\textit{Substep 1.1} At the start we deal with $S^{1}(f,g)$. Assumption \eqref{eq_hölder_meet_franke1} implies
\begin{align*}
\frac{1}{p} = \frac{1}{r} + \frac{s}{d} = \frac{1}{r_{1}} + \frac{1}{r_{2}} + \frac{s}{d} = \frac{1}{r_{1}} + \frac{1}{p_{2}} - \frac{s}{d} + \frac{s}{d} = \frac{1}{r_{1}} + \frac{1}{p_{2}}
\end{align*}
and \eqref{eq_hölder_meet_franke2} yields
\begin{align*}
\frac{1}{u} = \frac{1}{v} + \frac{s}{d} = \frac{1}{v_{1}} + \frac{1}{v_{2}} + \frac{s}{d} = \frac{1}{v_{1}} + \frac{1}{u_{2}} - \frac{s}{d} + \frac{s}{d} = \frac{1}{v_{1}} + \frac{1}{u_{2}} .
\end{align*}
Using Proposition \ref{prop_splitting_multipliers_simple} with $ \frac{1}{p} = \frac{1}{r_{1}} + \frac{1}{p_{2}}  $ and $ \frac{1}{u} = \frac{1}{v_{1}} + \frac{1}{u_{2}}  $ we get 
\begin{align*}
& \Vert  S^{1}(f,g) \vert  \mathcal{N}^{s}_{u,p,q}(\mathbb{R}^d)   \Vert  \\
& \qquad  = \Big ( \sum_{j = 0}^{\infty} 2^{jsq} \Vert \mathcal{F}^{-1}[\varphi_{j} \mathcal{F} S^{1}(f,g)]   \vert  \mathcal{M}^{u}_{p}(\mathbb{R}^d)  \Vert^{q}   \Big )^{\frac{1}{q}} \\
& \qquad  \lesssim  \Big ( \sum_{j = 0}^{\infty} 2^{jsq}  \Big \Vert  \sup_{k \in \mathbb{N}_{0}} \vert  f^{k}(x) \vert  \Big  \vert \mathcal{M}^{v_{1}}_{r_{1}}(\R)   \Big \Vert^{q}    \max_{\ell \in \{ -1, 0, 1, 2 \} }    \Vert  g_{j + \ell}   \vert \mathcal{M}^{u_{2}}_{p_{2}}(\R)  \Vert^{q}   \Big )^{\frac{1}{q}} .
\end{align*}
An application of Lemma \ref{lem_mor_eq_hardy} with $1 < r_{1} \leq v_{1} < \infty$ yields 
\begin{align*}
& \Vert  S^{1}(f,g) \vert  \mathcal{N}^{s}_{u,p,q}(\mathbb{R}^d)   \Vert  \\
& \qquad  \lesssim   \Vert  f   \vert \mathcal{M}^{v_{1}}_{r_{1}}(\R)   \Vert   \Big ( \sum_{j = 0}^{\infty} 2^{jsq}     \max_{\ell \in \{ -1, 0, 1, 2 \} }    \Vert  g_{j + \ell}   \vert \mathcal{M}^{u_{2}}_{p_{2}}(\R)  \Vert^{q}   \Big )^{\frac{1}{q}} \\
& \qquad  \leq   \Vert  f   \vert \mathcal{M}^{v_{1}}_{r_{1}}(\R)   \Vert   \Big ( \sum_{j = 0}^{\infty} 2^{jsq}   \Big (  \sum_{\ell = -1  }^{2}    \Vert  \mathcal{F}^{-1}[\varphi_{j + \ell} \mathcal{F}g]   \vert \mathcal{M}^{u_{2}}_{p_{2}}(\R)  \Vert \Big )^{q}   \Big )^{\frac{1}{q}} \\
& \qquad  \lesssim   \Vert  f   \vert \mathcal{M}^{v_{1}}_{r_{1}}(\R)   \Vert   \Big (   \sum_{\ell = -1 }^{2}  2^{- \ell s} \Big ( \sum_{j = 0}^{\infty} 2^{(j+\ell)sq}      \Vert  \mathcal{F}^{-1}[\varphi_{j + \ell} \mathcal{F}g]   \vert \mathcal{M}^{u_{2}}_{p_{2}}(\R)  \Vert^{q}   \Big )^{\frac{1}{q}}  \Big ) \\
& \qquad  \lesssim   \Vert  f   \vert \mathcal{M}^{v_{1}}_{r_{1}}(\R)   \Vert     \Vert g \vert  \mathcal{N}^{s}_{u_{2},p_{2},q}(\mathbb{R}^d)   \Vert. 
\end{align*}
Here we used $ \varphi_{n} := 0   $ for $n < 0$. Since $  r_{1} > 1  $, by Lemma \ref{l_bp1}(vi) and Lemma \ref{lem_FJ_2} we get $  \mathcal{N}^{s}_{u_{1},p_{1},r_{1}}(\mathbb{R}^{d}) \hookrightarrow  \mathcal{E}^{0}_{v_{1},r_{1},2}(\mathbb{R}^{d}) = \mathcal{M}^{v_{1}}_{r_{1}}(\mathbb{R}^d) $, where
\begin{align*}
s - \frac{d}{u_{1}} =  - \frac{d}{v_{1}}  \quad \mbox{and} \quad s > 0 \quad \mbox{and} \quad p_{1} < u_{1} \quad \mbox{and} \quad   s - \frac{d}{p_{1}}  = - \frac{d}{r_{1}}   ,
\end{align*}
what is ensured by \eqref{eq_hölder_meet_franke1} and \eqref{eq_hölder_meet_franke2}. In the special case $ p_{1} = u_{1}  $ we find $ r_{1} = v_{1}  $ and use \cite[Remark 3.3.5]{SiTri}. We obtain
\begin{equation}\label{eq_main_N_subs1.1}
\Vert  S^{1}(f,g) \vert  \mathcal{N}^{s}_{u,p,q}(\mathbb{R}^d)   \Vert \lesssim   \Vert  f   \vert   \mathcal{N}^{s}_{u_{1},p_{1},r_{1}}(\mathbb{R}^d) \Vert     \Vert g \vert  \mathcal{N}^{s}_{u_{2},p_{2},q}(\mathbb{R}^d)   \Vert. 
\end{equation}

\textit{Substep 1.2} To continue we investigate $S^{2}(f,g)$. Assumption \eqref{eq_hölder_meet_franke1} delivers
\begin{align*}
\frac{1}{p} = \frac{1}{r} + \frac{s}{d} = \frac{1}{r_{1}} + \frac{1}{r_{2}} + \frac{s}{d} = \frac{1}{p_{1}} - \frac{s}{d}  + \frac{1}{r_{2}} + \frac{s}{d} = \frac{1}{p_{1}}  + \frac{1}{r_{2}}
\end{align*}
and \eqref{eq_hölder_meet_franke2} yields
\begin{align*}
\frac{1}{u} = \frac{1}{v} + \frac{s}{d} = \frac{1}{v_{1}} + \frac{1}{v_{2}} + \frac{s}{d} = \frac{1}{u_{1}} - \frac{s}{d}  + \frac{1}{v_{2}} + \frac{s}{d} = \frac{1}{u_{1}}  + \frac{1}{v_{2}}.
\end{align*} 
Proposition \ref{prop_splitting_multipliers_simple} with $ \frac{1}{p} = \frac{1}{p_{1}} + \frac{1}{r_{2}}  $ and $ \frac{1}{u} = \frac{1}{u_{1}} + \frac{1}{v_{2}}  $ in combination with Lemma \ref{lem_mor_eq_hardy} with $ r_{2} > 1   $ implies
\begin{align*}
& \Vert  S^{2}(f,g) \vert  \mathcal{N}^{s}_{u,p,q}(\mathbb{R}^d)   \Vert  \\
& \qquad  = \Big ( \sum_{j = 0}^{\infty} 2^{jsq} \Vert \mathcal{F}^{-1}[\varphi_{j} \mathcal{F} S^{2}(f,g)]   \vert  \mathcal{M}^{u}_{p}(\mathbb{R}^d)  \Vert^{q}   \Big )^{\frac{1}{q}} \\
& \qquad  \lesssim  \Big ( \sum_{j = 0}^{\infty} 2^{jsq}  \Big \Vert  \sup_{k \in \mathbb{N}_{0}} \vert  g^{k}(x) \vert  \Big  \vert \mathcal{M}^{v_{2}}_{r_{2}}(\R)   \Big \Vert^{q}    \max_{\ell \in \{ -1, 0, 1, 2 \} }    \Vert  f_{j + \ell}   \vert \mathcal{M}^{u_{1}}_{p_{1}}(\R)  \Vert^{q}   \Big )^{\frac{1}{q}} \\
& \qquad \lesssim   \Vert  g   \vert \mathcal{M}^{v_{2}}_{r_{2}}(\R)   \Vert     \Vert f \vert  \mathcal{N}^{s}_{u_{1},p_{1},q}(\mathbb{R}^d)   \Vert. 
\end{align*}
Since $  r_{2} > 1  $, by Lemma \ref{l_bp1} and Lemma \ref{lem_FJ_2} we get $  \mathcal{N}^{s}_{u_{2},p_{2},r_{2}}(\mathbb{R}^{d}) \hookrightarrow  \mathcal{E}^{0}_{v_{2},r_{2},2}(\mathbb{R}^{d}) = \mathcal{M}^{v_{2}}_{r_{2}}(\mathbb{R}^d) $, where
\begin{align*}
s - \frac{d}{u_{2}} =  - \frac{d}{v_{2}}  \quad \mbox{and} \quad s > 0 \quad \mbox{and} \quad p_{2} < u_{2} \quad \mbox{and} \quad   s - \frac{d}{p_{2}}  = - \frac{d}{r_{2}}   .
\end{align*}
In the special case $ p_{2} = u_{2}  $ we find $ r_{2} = v_{2}  $ and use \cite[Remark 3.3.5]{SiTri}. We obtain
\begin{equation}\label{eq_main_N_subs1.2}
\Vert  S^{2}(f,g) \vert  \mathcal{N}^{s}_{u,p,q}(\mathbb{R}^d)   \Vert \lesssim   \Vert  g   \vert   \mathcal{N}^{s}_{u_{2},p_{2},r_{2}}(\mathbb{R}^d) \Vert     \Vert f \vert  \mathcal{N}^{s}_{u_{1},p_{1},q}(\mathbb{R}^d)   \Vert. 
\end{equation}

\textit{Substep 1.3} Now we deal with $S^{3}(f,g)$. Let $\frac{1}{t} = \frac{1}{p_{1}} + \frac{1}{p_{2}}$ and $\frac{1}{w} = \frac{1}{u_{1}} + \frac{1}{u_{2}}$. Let $\min(1,t) = \theta $. At first we consider the case $ q \leq \theta  $. By \eqref{eq_hölder_meet_franke2} we find
\begin{align*}
\frac{d}{w} - 2s = \frac{d}{u} - s .
\end{align*}
Consequently, Lemma \ref{lem_FJ_3} and  Proposition \ref{prop_splitting_multipliers}(ii),(iii) yield
\begin{align*}
& \Vert S^{3}(f,g) \vert  \mathcal{N}^{s}_{u,p,q}(\mathbb{R}^d)   \Vert \\
& \qquad \lesssim \Vert S^{3}(f,g) \vert  \mathcal{N}^{2s}_{w,t,q}(\mathbb{R}^d)   \Vert \\
& \qquad = \Big ( \sum_{j = 0}^{\infty} 2^{2s jq} \Vert \mathcal{F}^{-1}[\varphi_{j} \mathcal{F} S^{3}(f,g)]   \vert  \mathcal{M}^{w}_{t}(\mathbb{R}^d)  \Vert^{q}   \Big )^{\frac{1}{q}} \\
& \qquad \lesssim \Big ( \sum_{j = 0}^{\infty} 2^{2s jq} \max_{-1 \leq i \leq 1} \Big ( \sum_{\ell = -2}^{\infty} 2^{\ell d \theta ( \frac{1}{\theta} -  1 )} \Vert f_{j+ \ell} \vert \mathcal{M}^{u_{1}}_{p_{1}} (\mathbb{R}^d) \Vert^{\theta}  \Vert g_{j+ \ell + i} \vert \mathcal{M}^{u_{2}}_{p_{2}} (\mathbb{R}^d) \Vert^{\theta}  \Big )^{\frac{q}{\theta}}   \Big )^{\frac{1}{q}} \\
& \qquad \leq \Big (  \max_{-1 \leq i \leq 1}  \sum_{\ell = -2}^{\infty} 2^{\ell d q ( \frac{1}{\theta} -  1 )  }    \sum_{j = 0}^{\infty} 2^{2s jq}   \Vert f_{j+ \ell} \vert \mathcal{M}^{u_{1}}_{p_{1}} (\mathbb{R}^d) \Vert^{q}  \Vert g_{j+ \ell + i} \vert \mathcal{M}^{u_{2}}_{p_{2}} (\mathbb{R}^d) \Vert^{q}     \Big )^{\frac{1}{q}} \\
& \qquad = \Big (  \max_{-1 \leq i \leq 1}  \sum_{\ell = -2}^{\infty} 2^{\ell  q [ d ( \frac{1}{\theta} -  1 ) - 2s ]  }    \sum_{j = 0}^{\infty} 2^{2sq (\ell + j)  }   \Vert f_{j+ \ell} \vert \mathcal{M}^{u_{1}}_{p_{1}} (\mathbb{R}^d) \Vert^{q}  \Vert g_{j+ \ell + i} \vert \mathcal{M}^{u_{2}}_{p_{2}} (\mathbb{R}^d) \Vert^{q}     \Big )^{\frac{1}{q}} .
\end{align*}
By \eqref{eq_hölder_meet_franke1} we find 
\begin{align*}
d \Big ( \frac{1}{t} -  1  \Big ) = d \Big ( \frac{1}{p_{1}} + \frac{1}{p_{2}} -  1  \Big ) = d \Big ( \frac{1}{p} + \frac{s}{d} -  1  \Big ) < 2 s.
\end{align*}
This in combination with $s>0$ and the Hölder inequality yields
\begin{align*}
& \Vert S^{3}(f,g) \vert   \mathcal{N}^{s}_{u,p,q}(\mathbb{R}^d)   \Vert \\
& \quad \lesssim \Big (  \max_{-1 \leq i \leq 1}  \sum_{\ell = -2}^{\infty} 2^{\ell  q [ d ( \frac{1}{\theta} -  1 ) - 2s ]  }    \sum_{j = 0}^{\infty} 2^{sq (\ell + j)  }   \Vert f_{j+ \ell} \vert \mathcal{M}^{u_{1}}_{p_{1}} (\mathbb{R}^d) \Vert^{q}   2^{sq (\ell + j)  }  \Vert g_{j+ \ell + i} \vert \mathcal{M}^{u_{2}}_{p_{2}} (\mathbb{R}^d) \Vert^{q}     \Big )^{\frac{1}{q}} \\
& \quad \lesssim \max_{-1 \leq i \leq 1}   \Big (     \sum_{k = 0}^{\infty} 2^{ksq}   \Vert f_{k} \vert \mathcal{M}^{u_{1}}_{p_{1}} (\mathbb{R}^d) \Vert^{q}   2^{ksq  }  \Vert g_{k + i} \vert \mathcal{M}^{u_{2}}_{p_{2}} (\mathbb{R}^d) \Vert^{q}     \Big )^{\frac{1}{q}} \\
& \quad \leq   \Big (  \sum_{k = 0}^{\infty} 2^{2ksq}   \Vert f_{k} \vert \mathcal{M}^{u_{1}}_{p_{1}} (\mathbb{R}^d) \Vert^{2q} \Big )^{\frac{1}{2q}} \Big (  \sum_{k = 0}^{\infty}  2^{2ksq  }  \Vert g_{k} \vert \mathcal{M}^{u_{2}}_{p_{2}} (\mathbb{R}^d) \Vert^{2q}   \Big )^{\frac{1}{2q}}    \\
& \quad = \Vert f \vert  \mathcal{N}^{s}_{u_{1},p_{1},2q}(\mathbb{R}^d)   \Vert   \Vert g \vert  \mathcal{N}^{s}_{u_{2},p_{2},2q}(\mathbb{R}^d)   \Vert \\ 
& \quad \lesssim \Vert f \vert  \mathcal{N}^{s}_{u_{1},p_{1},q}(\mathbb{R}^d)   \Vert   \Vert g \vert  \mathcal{N}^{s}_{u_{2},p_{2},q}(\mathbb{R}^d)   \Vert .
\end{align*}
For $ q >  \theta  $ we proceed similarly, but use the triangle inequality to get 
\begin{align*}
& \Vert S^{3}(f,g) \vert  \mathcal{N}^{s}_{u,p,q}(\mathbb{R}^d)   \Vert^{\theta} \\
& \qquad \lesssim \Big ( \sum_{j = 0}^{\infty}  \max_{-1 \leq i \leq 1} \Big ( \sum_{\ell = -2}^{\infty} 2^{2s j \theta}  2^{\ell d \theta ( \frac{1}{\theta} -  1 )} \Vert f_{j+ \ell} \vert \mathcal{M}^{u_{1}}_{p_{1}} (\mathbb{R}^d) \Vert^{\theta}  \Vert g_{j+ \ell + i} \vert \mathcal{M}^{u_{2}}_{p_{2}} (\mathbb{R}^d) \Vert^{\theta}  \Big )^{\frac{q}{\theta}}   \Big )^{\frac{\theta}{q}} \\
& \qquad \lesssim  \max_{-1 \leq i \leq 1}  \sum_{\ell = -2}^{\infty}   \Big (   \sum_{j = 0}^{\infty}   2^{\ell d q ( \frac{1}{\theta} -  1 )  }  2^{2s jq}   \Vert f_{j+ \ell} \vert \mathcal{M}^{u_{1}}_{p_{1}} (\mathbb{R}^d) \Vert^{q}  \Vert g_{j+ \ell + i} \vert \mathcal{M}^{u_{2}}_{p_{2}} (\mathbb{R}^d) \Vert^{q}     \Big )^{\frac{\theta}{q}} \\
& \qquad =   \max_{-1 \leq i \leq 1}  \sum_{\ell = -2}^{\infty}   2^{\ell  \theta [ d ( \frac{1}{\theta} -  1 ) - 2s ]  }  \Big (  \sum_{j = 0}^{\infty} 2^{2sq (\ell + j)  }   \Vert f_{j+ \ell} \vert \mathcal{M}^{u_{1}}_{p_{1}} (\mathbb{R}^d) \Vert^{q}  \Vert g_{j+ \ell + i} \vert \mathcal{M}^{u_{2}}_{p_{2}} (\mathbb{R}^d) \Vert^{q}     \Big )^{\frac{\theta}{q}} \\
& \qquad \lesssim \Vert f \vert  \mathcal{N}^{s}_{u_{1},p_{1},q}(\mathbb{R}^d)   \Vert^{\theta}   \Vert g \vert  \mathcal{N}^{s}_{u_{2},p_{2},q}(\mathbb{R}^d)   \Vert^{\theta} .
\end{align*}

\textit{Substep 1.4} To complete Step 1, namely the proof of \eqref{eq_main_N_asequation}, we combine \eqref{eq_main_N_subs1.1}, \eqref{eq_main_N_subs1.2} and the result of Substep 1.3 to find
\begin{align*}
& \Vert f \cdot g \vert  \mathcal{N}^{s}_{u,p,q}(\mathbb{R}^d)   \Vert \\
& \qquad  \lesssim  \Vert  S^{1}(f,g) \vert  \mathcal{N}^{s}_{u,p,q}(\mathbb{R}^d)   \Vert + \Vert S^{2}(f,g) \vert  \mathcal{N}^{s}_{u,p,q}(\mathbb{R}^d)   \Vert + \Vert S^{3}(f,g) \vert  \mathcal{N}^{s}_{u,p,q}(\mathbb{R}^d)   \Vert \\
& \qquad \lesssim   \Vert  f   \vert   \mathcal{N}^{s}_{u_{1},p_{1},r_{1}}(\mathbb{R}^d) \Vert     \Vert g \vert  \mathcal{N}^{s}_{u_{2},p_{2},q}(\mathbb{R}^d)   \Vert + \Vert  g   \vert   \mathcal{N}^{s}_{u_{2},p_{2},r_{2}}(\mathbb{R}^d) \Vert     \Vert f \vert  \mathcal{N}^{s}_{u_{1},p_{1},q}(\mathbb{R}^d)   \Vert  \\
& \qquad \qquad \qquad +  \Vert f \vert  \mathcal{N}^{s}_{u_{1},p_{1},q}(\mathbb{R}^d)   \Vert   \Vert g \vert  \mathcal{N}^{s}_{u_{2},p_{2},q}(\mathbb{R}^d)   \Vert \\
& \qquad \lesssim \Vert f \vert  \mathcal{N}^{s}_{u_{1},p_{1},q_{1}}(\mathbb{R}^d)   \Vert   \Vert g \vert  \mathcal{N}^{s}_{u_{2},p_{2},q_{2}}(\mathbb{R}^d)   \Vert .
\end{align*}
Here we used \eqref{eq_main_N_cond_qq1q2}. This proves \eqref{eq_main_N_asequation}.

\textit{Step 2.} Now we deal with the Triebel-Lizorkin-Morrey spaces and prove (ii). Let $  f   \in   \mathcal{E}^{s}_{u_{1},p_{1},q_{1}}(\mathbb{R}^d) $ and $ g \in  \mathcal{E}^{s}_{u_{2},p_{2},q_{2}}(\mathbb{R}^d) $.   Again, by \eqref{splitting_1} - \eqref{splitting_1_new} we find
\begin{align*}
\Vert f \cdot g \vert  \mathcal{E}^{s}_{u,p,q}(\mathbb{R}^d)   \Vert & =   \Vert  S^{1}(f,g) + S^{2}(f,g) + S^{3}(f,g) \vert  \mathcal{E}^{s}_{u,p,q}(\mathbb{R}^d)   \Vert \\
& \lesssim  \Vert  S^{1}(f,g) \vert  \mathcal{E}^{s}_{u,p,q}(\mathbb{R}^d)   \Vert + \Vert S^{2}(f,g) \vert  \mathcal{E}^{s}_{u,p,q}(\mathbb{R}^d)   \Vert + \Vert S^{3}(f,g) \vert  \mathcal{E}^{s}_{u,p,q}(\mathbb{R}^d)   \Vert .
\end{align*}

\textit{Substep 2.1} At first we deal with $S^{1}(f,g)$. Assumptions \eqref{eq_hölder_meet_franke1} and \eqref{eq_hölder_meet_franke2} imply
\begin{align*}
\frac{1}{p}  = \frac{1}{r_{1}} + \frac{1}{p_{2}}  \qquad \mbox{and} \qquad \frac{1}{u}  = \frac{1}{v_{1}} + \frac{1}{u_{2}} .
\end{align*}
Using a slightly modified version of Proposition \ref{prop_splitting_multipliers}(i) with $ \frac{1}{p} = \frac{1}{r_{1}} + \frac{1}{p_{2}}  $ and $ \frac{1}{u} = \frac{1}{v_{1}} + \frac{1}{u_{2}}  $ we get 
\begin{align*}
\Vert  S^{1}(f,g) \vert  \mathcal{E}^{s}_{u,p,q}(\mathbb{R}^d)   \Vert  
&  = \Big \Vert \Big ( \sum_{k = 0}^{\infty} 2^{ksq} \Big  \vert \mathcal{F}^{-1}\Big [\varphi_{k} \mathcal{F}S^{1}(f,g) \Big ](\cdot) \Big \vert^{q} \Big )^{\frac{1}{q}} \Big \vert \mathcal{M}^{u}_{p}(\R) \Big \Vert \\
&  \lesssim \Big \Vert  \sup_{k \in \mathbb{N}_{0}} \vert  f^{k}(x) \vert   \Big \vert \mathcal{M}^{v_{1}}_{r_{1}}(\R) \Big \Vert \Big \Vert \Big ( \sum_{k = 0}^{\infty} 2^{ksq}  \vert g_{k}(\cdot) \vert^{q} \Big )^{\frac{1}{q}} \Big \vert \mathcal{M}^{u_{2}}_{p_{2}}(\R) \Big \Vert .
\end{align*}
Lemma \ref{lem_mor_eq_hardy} with $1 < r_{1} \leq v_{1} < \infty$ and \eqref{eq_f_lowerindex} yield 
\begin{align*}
\Vert  S^{1}(f,g) \vert  \mathcal{E}^{s}_{u,p,q}(\mathbb{R}^d)   \Vert    \lesssim  \Vert  f  \vert \mathcal{M}^{v_{1}}_{r_{1}}(\R)  \Vert \Vert  g \vert  \mathcal{E}^{s}_{u_{2},p_{2},q}(\mathbb{R}^d)   \Vert .
\end{align*}
Since $  r_{1} > 1  $, by Lemma \ref{l_bp1}(vi) and Lemma \ref{lem_FJ_4} we get $  \mathcal{E}^{s}_{u_{1},p_{1},q}(\mathbb{R}^{d}) \hookrightarrow  \mathcal{E}^{0}_{v_{1},r_{1},2}(\mathbb{R}^{d}) = \mathcal{M}^{v_{1}}_{r_{1}}(\mathbb{R}^d) $, where we used 
\begin{align*}
s - \frac{d}{u_{1}} =  - \frac{d}{v_{1}}  \qquad \mbox{and} \qquad u_{1} < v_{1} \qquad \mbox{and} \qquad    \frac{r_{1}}{v_{1}}  \leq  \frac{p_{1}}{u_{1}}   ,
\end{align*}
what is ensured by $s > 0$ and \eqref{eq_hölder_meet_franke1} and \eqref{eq_hölder_meet_franke2}. This in combination with \eqref{eq_main_E_cond_qq1q2} implies
\begin{equation}\label{eq_main_E_subs1.1}
\Vert  S^{1}(f,g) \vert  \mathcal{E}^{s}_{u,p,q}(\mathbb{R}^d)   \Vert \lesssim   \Vert  f   \vert   \mathcal{E}^{s}_{u_{1},p_{1},q_{1}}(\mathbb{R}^d) \Vert     \Vert g \vert  \mathcal{E}^{s}_{u_{2},p_{2},q_{2}}(\mathbb{R}^d)   \Vert. 
\end{equation}

\textit{Substep 2.2} To continue we investigate $S^{2}(f,g)$. Using \eqref{eq_S2=S1} and Proposition \ref{prop_splitting_multipliers}(i) with $ \frac{1}{p} = \frac{1}{p_{1}} + \frac{1}{r_{2}}  $ and $ \frac{1}{u} = \frac{1}{u_{1}} + \frac{1}{v_{2}}  $ we get 
\begin{align*}
\Vert  S^{2}(f,g) \vert  \mathcal{E}^{s}_{u,p,q}(\mathbb{R}^d)   \Vert  
&  = \Big \Vert \Big ( \sum_{k = 0}^{\infty} 2^{ksq} \Big  \vert \mathcal{F}^{-1}\Big [\varphi_{k} \mathcal{F}S^{1}(g,f) \Big ](\cdot) \Big \vert^{q} \Big )^{\frac{1}{q}} \Big \vert \mathcal{M}^{u}_{p}(\R) \Big \Vert \\
&   \lesssim \Big \Vert  \sup_{k \in \mathbb{N}_{0}} \vert  g^{k}(x) \vert   \Big \vert \mathcal{M}^{v_{2}}_{r_{2}}(\R) \Big \Vert \Big \Vert \Big ( \sum_{k = 0}^{\infty} 2^{ksq}  \vert f_{k}(\cdot) \vert^{q} \Big )^{\frac{1}{q}} \Big \vert \mathcal{M}^{u_{1}}_{p_{1}}(\R) \Big \Vert .
\end{align*}
Lemma \ref{lem_mor_eq_hardy} with $1 < r_{2} \leq v_{2} < \infty$ and \eqref{eq_f_lowerindex} imply 
\begin{align*}
\Vert  S^{2}(f,g) \vert  \mathcal{E}^{s}_{u,p,q}(\mathbb{R}^d)   \Vert    \lesssim  \Vert  g  \vert \mathcal{M}^{v_{2}}_{r_{2}}(\R)  \Vert \Vert  f \vert  \mathcal{E}^{s}_{u_{1},p_{1},q}(\mathbb{R}^d)   \Vert .
\end{align*}
As before, since $  r_{2} > 1  $, by Lemma \ref{l_bp1}(vi) and Lemma \ref{lem_FJ_4} we get $  \mathcal{E}^{s}_{u_{2},p_{2},q}(\mathbb{R}^{d}) \hookrightarrow  \mathcal{E}^{0}_{v_{2},r_{2},2}(\mathbb{R}^{d}) = \mathcal{M}^{v_{2}}_{r_{2}}(\mathbb{R}^d) $. An application of \eqref{eq_main_E_cond_qq1q2} delivers
\begin{equation}\label{eq_main_E_subs2.2}
\Vert  S^{2}(f,g) \vert  \mathcal{E}^{s}_{u,p,q}(\mathbb{R}^d)   \Vert \lesssim   \Vert  f   \vert   \mathcal{E}^{s}_{u_{1},p_{1},q_{1}}(\mathbb{R}^d) \Vert     \Vert g \vert  \mathcal{E}^{s}_{u_{2},p_{2},q_{2}}(\mathbb{R}^d)   \Vert. 
\end{equation}

\textit{Substep 2.3} Now we deal with $S^{3}(f,g)$. As in Step 1 let $\frac{1}{t} = \frac{1}{p_{1}} + \frac{1}{p_{2}}$ and $\frac{1}{w} = \frac{1}{u_{1}} + \frac{1}{u_{2}}$. An application of Lemma \ref{lem_FJ_4} yields $ \mathcal{E}^{2s}_{w,t,\infty}(\mathbb{R}^d) \hookrightarrow   \mathcal{E}^{s}_{u,p,q}(\mathbb{R}^d)  $. To see this we used $s > 0$ and \eqref{eq_hölder_meet_franke1} as well as \eqref{eq_hölder_meet_franke2}. Consequently, Proposition \ref{prop_splitting_multipliers}(iv) yields
\begin{align*}
 \Vert S^{3}(f,g) \vert  \mathcal{E}^{s}_{u,p,q}(\mathbb{R}^d)   \Vert  
&  \lesssim \Vert S^{3}(f,g) \vert  \mathcal{E}^{2s}_{w,t, \infty}(\mathbb{R}^d)   \Vert \\
&  = \Big \Vert  \sup_{k \in \mathbb{N}_{0}}  2^{2ks} \Big \vert \mathcal{F}^{-1}\Big [\varphi_{k} \mathcal{F}S^{3}(f,g) \Big ](\cdot) \Big \vert   \Big \vert \mathcal{M}^{w}_{t}(\R) \Big \Vert \\
&   \lesssim \max_{-1 \leq j \leq 1}  \Big \Vert  \sup_{k \in \mathbb{N}_{0}} 2^{ks} \vert f_{k} \vert  \Big \vert \mathcal{M}^{u_{1}}_{p_{1}} (\mathbb{R}^d) \Big \Vert \Big  \Vert   \sup_{k \in \mathbb{N}_{0}} 2^{ks} \vert g_{k + j} \vert  \Big \vert \mathcal{M}^{u_{2}}_{p_{2}} (\mathbb{R}^d) \Big  \Vert .
\end{align*}
Here we used $2s > \sigma_{t} $, which follows by \eqref{eq_hölder_meet_franke1}. By \eqref{eq_f_lowerindex} we get 
\begin{align*}
& \Vert S^{3}(f,g) \vert  \mathcal{E}^{s}_{u,p,q}(\mathbb{R}^d)   \Vert  \\
& \qquad  \lesssim \max_{-1 \leq j \leq 1}  \Big \Vert  \sup_{k \in \mathbb{N}_{0}} 2^{ks} \vert \mathcal{F}^{-1}[\varphi_{k} \mathcal{F}f] \vert  \Big \vert \mathcal{M}^{u_{1}}_{p_{1}} (\mathbb{R}^d) \Big \Vert \Big  \Vert   \sup_{k \in \mathbb{N}_{0}} 2^{ks} \vert \mathcal{F}^{-1}[\varphi_{k+j} \mathcal{F}g] \vert  \Big \vert \mathcal{M}^{u_{2}}_{p_{2}} (\mathbb{R}^d) \Big  \Vert \\
& \qquad \lesssim \Vert f \vert  \mathcal{E}^{s}_{u_{1},p_{1},\infty}(\mathbb{R}^d)   \Vert  \Vert g \vert  \mathcal{E}^{s}_{u_{2},p_{2},\infty}(\mathbb{R}^d)   \Vert \\
& \qquad \lesssim \Vert f \vert  \mathcal{E}^{s}_{u_{1},p_{1},q_{1}}(\mathbb{R}^d)   \Vert  \Vert g \vert  \mathcal{E}^{s}_{u_{2},p_{2},q_{2}}(\mathbb{R}^d)   \Vert .
\end{align*}
Here we applied \eqref{eq_main_E_cond_qq1q2}.

\textit{Substep 2.4} To complete the proof we combine \eqref{eq_main_E_subs1.1}, \eqref{eq_main_E_subs2.2} and Substep 2.3. Then we find 
\begin{align*}
& \Vert f \cdot g \vert  \mathcal{E}^{s}_{u,p,q}(\mathbb{R}^d)   \Vert \\   
& \qquad  \lesssim  \Vert  S^{1}(f,g) \vert  \mathcal{E}^{s}_{u,p,q}(\mathbb{R}^d)   \Vert + \Vert S^{2}(f,g) \vert  \mathcal{E}^{s}_{u,p,q}(\mathbb{R}^d)   \Vert + \Vert S^{3}(f,g) \vert  \mathcal{E}^{s}_{u,p,q}(\mathbb{R}^d)   \Vert \\
& \qquad  \lesssim  \Vert  f   \vert   \mathcal{E}^{s}_{u_{1},p_{1},q_{1}}(\mathbb{R}^d) \Vert     \Vert g \vert  \mathcal{E}^{s}_{u_{2},p_{2},q_{2}}(\mathbb{R}^d)   \Vert .
\end{align*}
This is the desired result. Hence, the proof of Theorem \ref{thm_main_hölder1} is complete. 
\hfill \qedsymbol

\end{document}